\documentclass[english,reqno]{amsart}
\usepackage[english]{babel}
\usepackage{amsmath,amsthm,amssymb}
\usepackage{latexsym}
\usepackage{enumerate}
\usepackage[latin1]{inputenc}
\usepackage{layout}
\usepackage{type1cm}
\usepackage{hyperref}
\usepackage[pagewise]{lineno}%\linenumbers    % numera as linhas 

\usepackage{color}

\newtheorem{theorem}{Theorem}
\newtheorem{lemma}{Lemma}
\newtheorem{proposition}{Proposition}

\theoremstyle{remark}
\newtheorem{remark}{Remark}

\newcommand{\rafa}[1]{{\color{black} #1}}
\newcommand{\victor}[1]{{\color{black} #1}}
\newcommand{\ra}[1]{{\color{black} #1}}

\title[General one dimensional Spin Models with Markovian Structure]{Existence of Gibbs states and maximizing measures on a general one-dimensional lattice system with markovian structure} 
\author[Rafael Rig\~ao Souza and Victor Vargas]{}

\thanks{The second author is supported by FFJC-MINCIENCIAS Process 80740-628-2020.}

\begin{document}
\maketitle

\centerline{\scshape Rafael Rig\~ao Souza}
\medskip
{\footnotesize
 \centerline{Department of Mathematics - IME-UFRGS}
   \centerline{91509-900 Porto Alegre - Brazil}
	   \centerline{rafars@mat.ufrgs.br}
}

\medskip
\centerline{\scshape Victor Vargas}
\medskip
{\footnotesize
 \centerline{School of Mathematics - National University of Colombia}
   \centerline{050034 Medell\'in - Colombia}
	   \centerline{vavargasc@unal.edu.co}
}

\begin{abstract}
Consider a compact metric space $(M, d_M)$ and $X = M^{\mathbb{N}}$. We prove a Ruelle's Perron Frobenius Theorem for a class of compact subshifts with Markovian structure introduced in [Bull. Braz. Math. Soc. 45 (2014), pp. 53-72] which are defined from a continuous function $A : M \times M \to \mathbb{R}$ that determines the set of admissible sequences. \ra{In particular, this class of subshifts includes the finite Markov shifts and models where the alphabet is given by the unit circle $S^1$}. \rafa{Using the involution Kernel, we characterize the normalized eigenfunction of the Ruelle operator associated to its maximal eigenvalue and present an extension of its corresponding Gibbs state to the bilateral approach.} \ra{ From these results, we  prove existence} of equilibrium states and accumulation points at zero temperature in a particular class of countable Markov shifts.   
\end{abstract}
\victor{

}

\maketitle

\section{Introduction}
\label{introduction-section}

The  thermodynamic \victor{formalism has} its origins in the second half of the XX century with the study of problems of minimization of energy in classical mechanics. At that moment, some mathematicians interested in the study of these problems, among them, \victor{Yakov Sinai, David Ruelle and Rufus Bowen, adopted a concept known as Gibbs state from the theoretical physics setting to the ergodic theory (see for instance \cite{zbMATH03482645}, \cite{MR0234697}, \cite{MR511655} and \cite{MR691854}). The so called Gibbs states usually} represent  observables optimizing the \victor{free energy} on a system of particles modeled on a lattice with interactions described by an observable satisfying certain regularity conditions. From a mathematical point of view such Gibbs states are Borelian probability measures defined on a lattice, which \victor{can be} obtained from the eigenvalues and eigenvectors of a transfer operator associated to the \victor{observable that represents the interactions.}

Different approaches to this theory have been studied by many authors in several contexts of  symbolic dynamics, both in compact and non-compact settings. In the seminal work \cite{MR0234697} the thermodynamic formalism on uni-dimensional lattices was presented. That work introduced a useful tool called Ruelle operator, also known in the literature as transfer operator, \victor{which is still today} one of the most important tools used to find \victor{such states that minimize the free energy of the system.} Some years later, in \cite{MR1085356}, these problems were studied in a more interesting dynamical context, known as finite Markov shifts under hypothesis of aperiodicity in the incidence matrix. \victor{In \cite{MR1853808} and \cite{MR1738951}} these results were generalized for the non-compact setting of countable Markov shifts, \victor{both in the topologically mixing and the finitely irreducible case.} \rafa{Another interesting setting, \ra{where the alphabet is given by the unit circle $S^1$, or the unit interval $[0,1]$,} was studied in   \cite{zbMATH05995684} and later in the seminal paper} \victor{about limits at zero temperature in \cite{MR2316198}.} \rafa{After,  \cite{MR2864625} and \cite{MR3377291} considered the case where the alphabet is any compact metric space. We also have some interesting generalizations for bounded Polish metric spaces in \cite{CSS19,1LoVa19} and even non-bounded Polish metric spaces in a linear dynamical approach in \cite{LMSV19}.} 

Among the main utilities of the study of transfer operators are their multiple applications in problems of ergodic optimization using techniques of selection and non-selection at zero temperature, \victor{ as we can see in \cite{MR2316198}  (see also \cite{MR2496111})}. 

\rafa{ Another very important work } in that direction was presented in \cite{MR1958608}, in which was guaranteed uniqueness of the ground state associated to a locally constant potential in the setting of finite Markov shifts. After that, some interesting techniques of renormalization were introduced, allowing to find explicit expressions of the ground state (see for details \cite{MR2818689,MR2176962}). From a non-compact point of view, in \cite{MR2151222} was proved existence of maximizing measures in the context of countable Markov shifts satisfying the finitely primitive condition. After that, in \cite{MR3864383}, a generalization of that result in the case of topologically transitive countable Markov shifts was presented. The uniqueness of the ground state was proved in \cite{MR2800665} in the setting of countable Markov shifts satisfying the BIP property, however, that problem is still open for the topologically mixing case. \ra{When the alphabet is given by the unit circle $S^1$}, problems of selection and non-selection at zero temperature were studied \victor{ in \cite{MR2864625}, \cite{MR2496111}  and \cite{MR2316198} } for the classical approach on the interval $[0, 1]$, and these results were generalized to compact metric spaces in \cite{MR3377291} and to non-compact bounded Polish metric spaces in \cite{CSS19} and \cite{1LoVa19}.

In this paper we present the thermodynamical formalism in a symbolic dynamical context introduced in \victor{\cite{MR3194082} using} approaches similar to the ones that appear in \cite{MR2864625} and \cite{MR1085356}, we are able to prove a Ruelle´s Perron-Frobenius Theorem in our setting, both in the topologically transitive case and the topologically mixing case. Furthermore, using some techniques developed in \cite{1LoVa19}, we guarantee existence of Gibbs states and maximizing measures in an interesting particular case of countable Markov shifts with irreducible incidence matrix that satisfy the BP property (see for instance \cite{ChFr19}), but which is not immersed in the class of countable Markov shifts satisfying either the BIP property or the finitely primitive condition \victor{(see also \cite{MR2151222}, \cite{MR2800665} and \cite{MR1955261}).} 

\rafa{On other hand, the so called involution kernel appears as a useful technique to characterize the eigenfunction associated to the maximal eigenvalue of the Ruelle operator in terms of the eigenprobability and the potential defining the operator} (see for instance \cite{MR2864625,MR3297616}). \victor{An interesting case where are presented explicit expressions of the eigenprobability in terms of a potential $\varphi$ satisfying $\mathcal{L}_\varphi(1) = 1$, can be found in \cite{MR3656287} and \cite{Moh18}.} \rafa{ In this work, we present an involution kernel adapted to our matter and we use that to find an expression of the eigenfunction in terms of the 
	involution kernel and the eigenprobability,  and also an optimal transport measure between the Gibbs state associated to a potential satisfying suitable conditions and its corresponding dual.}

This paper is organized as follows: In section \ref{main-results-section} we state the main results of the paper and are included some definitions. In section \ref{theory-PF-section} we present the proofs of Theorem \ref{RPF-theorem} and Proposition \ref{RPF-theorem-SFT}. We also state a variational principle in order to prove Proposition \ref{ground-states}. Finally, in section \ref{involution-kernel-section} we present the proof of Theorem \ref{bilateral-extension}.

\section{Main Results}
\label{main-results-section}

Consider a compact metric space $(M, d_M)$ and define $X$ as the set of sequences taking values in $M$ (the set \rafa{$M$ sometimes is called the alphabet)}. As a consequence of the Tychonoff's Theorem, the set $X$ equipped with the metric 

\begin{equation}
\label{metric}
d(x, y) := \sum_{n = 1}^{\infty}\frac{1}{2^n}d_M(x_n, y_n) \,,
\end{equation}
results in a compact metric space. The {\bf shift map} is defined as the function $\sigma : X \to X$ given by $\sigma((x_n)_{n \in\mathbb{N}}) = (x_{n+1})_{n \in \mathbb{N}}$. \rafa{We consider the Bernoulli system of sequences in $X$ with the shift map acting on it and a suitable potential $\varphi$ from $X$ into $\mathbb{R}$ determining the interactions on the system.} The thermodynamical formalism on this class of models has been widely studied in the setting of compact metric spaces (see for instance \cite{MR2864625, MR3377291, MR2496111}), as well as in the non-compact setting of bounded Polish metric spaces (see for details \cite{CSS19, 1LoVa19}).  

In this paper we stress that condition, studying the thermodynamical formalism on a class of subshifts introduced in \cite{MR3194082} \victor{(see also \cite{LW20} and \cite{LW19}),} in which only some of the sequences belonging to the set $X$ are allowed. 

The set of admissible sequences is characterized in the following way: consider a continuous function $A : M \times M \to \mathbb{R}$ and a compact set $I \subset \mathbb{R}$. We say that a sequence $x = (x_n)_{n \in \mathbb{N}} \in X$ is an {\bf admissible sequence} associated to the map $A$ and the set $I$, if $A(x_n, x_{n + 1}) \in I$ for each $n \in \mathbb{N}$. Through this paper we will denote the set of such admissible sequences by $\mathcal{B}(A, I)$. Given an element $b \in M$, we define the {\bf section of $b$ in $A^{-1}(I)$} as the set of elements $a \in M$ such that $A(a, b) \in I$, which will be denoted by $s(b)$. It is not difficult to check that the continuity of the map $A$ implies that $s(b)$ is a compact subset of $M$ for each $b\in M$.
%\footnote{\rafa{Vitor me parece que este detalhamento a seguir é desnecessário. (Mas se quiseres manter a frase a seguir, não me oponho): By continuity of the map $A$, for any convergent sequence $(a_n)_{n \in \mathbb{N}}$ of elements in $s(b)$ with $\lim_{n \in \mathbb{N}}a_n = a$, it follows that $A(a, b) = \lim_{n \in \mathbb{N}}A(a_n, b) \in I$, that is, $a \in s(b)$, which implies that $s(b)$ is a closed and thus a compact subset of $M$.}}% 
By the above, we can define a map $s : M \to \mathcal{K}(M)$ that assigns to each $b \in M$ its corresponding section $s(b) \in \mathcal{K}(M)$, where $\mathcal{K}(M)$ denotes the collection of all the compact subsets of $M$ equipped with the Hausdorff metric \victor{(see for details \cite{MR3194082})}. 
%(note that $s(b)$ is a compact subset of $M$)

From now on, we will assume that the map $A$ is such that $s$ results in a locally constant function, that is, for any $b \in M$, there is an open neighborhood $U_b \subset M$ containing $b$ such that $s(b') = s(b)$ for each $b' \in U_b$.

It is easy to check that $\mathcal{B}(A, I)$ is $\sigma$-invariant, moreover, in \cite{MR3194082} it was showed that $\mathcal{B}(A, I) \subset X$ is a closed metric subspace (when it is equipped with the metric induced by the metric $d$ defined in \eqref{metric}). Therefore, $\mathcal{B}(A, I)$ results in a topological subshift of $M^{\mathbb{N}}$. 

If $M$ is a connected compact metric space, for instance $M = [0, 1]$, it follows that $s : M \to \mathcal{K}(M)$ is a constant function, which reduce our approach \ra{ to the classical model on the lattice $[0, 1]^{\mathbb{N}}$ (see for instance \cite{MR2864625})}. In the case that the function $A$ is a constant map taking the value $\kappa$ and $I := \{ \kappa \} = A(M \times M)$, our setting is the same as the one studied in \cite{MR3377291}. On other hand, if $M = \{1, \ldots, d\}$, $I = \{1\}$, ${\bf A} \in M_{d \times d}(\{0, 1\})$ and the function $A$ is defined by $A(i, j) = 1$ if and only if ${\bf A}_{i, j} = 1$ and $A(i, j) \neq 1$ if and only if ${\bf A}_{i, j} = 0$, it follows that $\mathcal{B}(A, \{1\})$ is a finite Markov shift with incidence matrix ${\bf A}$ on the alphabet $\{1, \ldots, d\}$ (see \cite{MR3114331} and \cite{MR1085356}).

Let $Y$ be a subset of $X$. We know that $Y$ is a metric subspace of $X$, with the metric induced by the metric $d$ defined in \eqref{metric}. We will use the notation $\mathcal{C}_b(Y)$ for the set of bounded continuous functions from $Y$ into $\mathbb{R}$ equipped with the norm $\| \cdot \|_{\infty}$ given by $\|\varphi\|_{\infty} = \sup\{|\varphi(x)| : x \in Y\}$. 
When $Y$ is a compact metric subspace of $X$, we will denote by $\mathcal{C}(Y)$ the set of continuous functions from $Y$ into $\mathbb{R}$. We will also use the notation $\mathcal{H}_{\alpha}(Y)$ for the set of $\alpha$-H\"older continuous functions from $Y$ into $\mathbb{R}$ equipped with the norm $\| \cdot \|_{\alpha}$ given by $\|\varphi\|_{\alpha} = \|\varphi\|_{\infty} + \mathrm{Hol}_{\varphi}$, where 
\[
\mathrm{Hol}_{\varphi} := \sup\Bigl\{\frac{|\varphi(x) - \varphi(y)|}{d(x, y)^{\alpha}} : x \neq y ,\, x, y \in Y \Bigr\}.
\]

It is widely known that all the spaces of functions mentioned above result in Banach spaces when they are equipped with its corresponding norms.

Now we will define the Ruelle operator associated to a map $\varphi \in \mathcal{C}(\mathcal{B}(A, I))$ in our present setting. Fixing a Borelian a priori probability measure $\nu$ on $M$ and assuming that $\nu$ has full support, we define the {\bf generalized Ruelle operator} associated to $\varphi$ as the map assigning to each $\psi \in \mathcal{C}(\mathcal{B}(A, I))$ the function $\mathcal{L}_{\varphi}(\psi)$ given by
\begin{equation}
\label{Ruelle-operator}
\mathcal{L}_{\varphi}(\psi)(x) := \int_{s(x_1)}e^{\varphi(ax)}\psi(ax)d\nu(a) \,, 
\end{equation} 
where $ax \in \mathcal{B}(A, I)$ is the concatenation of the word $a \in s(x_1)$ and the sequence $x \in \mathcal{B}(A, I)$.

By the above, it follows that for each $n \in \mathbb{N}$, the $n$-th iterate of the Ruelle operator is given by the map assigning to each $\psi \in \mathcal{C}(\mathcal{B}(A, I))$ the function
\[
\mathcal{L}^n_{\varphi}(\psi)(x) = \int_{s(a_{n-1})} \ldots \int_{s(x_1)}e^{S_n\varphi(a^nx)}\psi(a^nx)d\nu(a_1) \ldots d\nu(a_n)\,, 
\] 
where $S_n \varphi(y) = \sum^{n-1}_{j=0}\varphi(\sigma^j(y))$ and each $a^n = a_n \ldots a_1$ is a word of length $n$ satisfying that the concatenation $a^nx \in \mathcal{B}(A, I)$, which is equivalent to say that $a_1 \in s(x_1), a_2 \in s(a_1), \ldots, a_n \in s(a_{n-1})$ and $x \in \mathcal{B}(A, I)$. 

In \cite{MR3194082} it was proved that this operator is well defined, that is, the integral in the right side of the above equation is finite for each $x \in \mathcal{B}(A, I)$. Furthermore, it is easy to check, using the fact that continuous functions are uniformly continuous on compact sets, that  $\mathcal{L}_{\varphi}$ preserves the set of functions $\mathcal{C}(\mathcal{B}(A, I))$.
	
In the case where the potential $\varphi$ belongs to $\mathcal{H}_{\alpha}(\mathcal{B}(A, I))$, we have that the set $\mathcal{H}_{\alpha}(\mathcal{B}(A, I))$ is preserved by the Ruelle operator $\mathcal{L}_{\varphi}$. 

Indeed, if we have $\psi \in \mathcal{H}_{\alpha}(\mathcal{B}(A, I))$, it follows that for any pair $x, y \in \mathcal{B}(A, I)$ such that $x_1 = y_1$, we have 
\begin{align}
\Bigl|\mathcal{L}_{\varphi}(\psi)(x) - \mathcal{L}_{\varphi}(\psi)(y) \Bigr|
&\leq \int_{s(x_1)} \Bigl| e^{\varphi(ax)}\psi(ax) - e^{\varphi(ay)}\psi(ay) \Bigr| d\nu(a) \nonumber \\
&\leq \frac{1}{2^{\alpha}}\Bigl(\mathrm{Hol}_{e^\varphi}\|\psi\|_{\infty} + \mathrm{Hol}_{\psi}e^{\|\varphi\|_{\infty}}\Bigr)d(x, y)^{\alpha} \nonumber \,.
\end{align}

Thus, under the assumption that the function $s$ is locally constant, we conclude that the function $\mathcal{L}_{\varphi}(\psi)$ is locally H\"older continuous, which implies that $\mathcal{L}_{\varphi}(\psi) \in \mathcal{H}_{\alpha}(\mathcal{B}(A, I))$ by compactness of the set $\mathcal{B}(A, I)$ (A more detailed explanation about this claim appears in the proof of Theorem \ref{RPF-theorem}). 

In order to simplify reading of this text, for any Borelian measure $\mu$ defined on a metric subspace $Y \subset X$ and any $\psi \in \mathcal{C}(Y)$, we will use the notation 
\[
\mu(\psi) := \int_{Y} \psi d\mu \,.
\]

When necessary, we will use the notation
\[
\mu(\psi(x)) := \int_{Y} \psi(x) d\mu(x) \,.
\]

We say that a Borelian measure $\mu$ defined on the metric subspace $Y \subset X$ is a {\bf$\sigma$-invariant measure}, if for any Borelian set $E \subset Y$ we have  $\mu(\sigma^{-1}(E)) = \mu(E)$. Through this paper, we will use the notation $\mathcal{M}_{\sigma}(Y)$ for the set of all the $\sigma$-invariant probability measures on $Y$.

From the properties of the dual of a Banach space, we can define the {\bf dual Ruelle operator} associated to a potential $\varphi \in \mathcal{H}_{\alpha}(\mathcal{B}(A, I))$, as the map $\mathcal{L}^*_{\varphi}$ assigning to each Radon measure $\mu$ on the Borelian sets of $\mathcal{B}(A, I)$, the Radon measure $\mathcal{L}^*_{\varphi}(\mu)$, which is defined as the Radon measure satisfying for each $\psi \in \mathcal{C}(\mathcal{B}(A, I))$ the following equation 
\begin{equation}
\label{dual-Ruelle-operator}
\mathcal{L}^*_{\varphi}(\mu)(\psi) := \mu(\mathcal{L}_{\varphi}(\psi)) \,. 
\end{equation}

In a similar way, for each $n \in \mathbb{N}$ we define the $n$-th iterate of $\mathcal{L}^*_{\varphi}$ as the operator assigning to each Radon measure $\mu$, the Radon measure $\mathcal{L}^{*,n}(\mu)$, satisfying for each $\psi \in \mathcal{C}(\mathcal{B}(A, I))$ the following equation 
\[
\mathcal{L}^{*,n}_{\varphi}(\mu)(\psi) = \mu(\mathcal{L}^n_{\varphi}(\psi)) \,. 
\]

Note that by completeness and separability of the metric space $\mathcal{B}(A, I)$, the operators $\mathcal{L}^*_{\varphi}$ and $\mathcal{L}^{*,n}_{\varphi}$ are in fact defined on the set of all the Borelian measures on $\mathcal{B}(A, I)$.

We say that a metric subspace $Y \subset X$ is {\bf topologically transitive}, if for any pair of open sets $U, V \subset Y$, there exists $n \in \mathbb{N}$ such that $\sigma^{-n}(U) \cap V \neq \emptyset$. In addition, we say that $Y \subset X$ is {\bf topologically mixing}, if for any pair of open sets $U, V \subset Y$, there exists $n \in \mathbb{N}$ such that $\sigma^{-m}(U) \cap V \neq \emptyset$ for each $m \geq n$.

Under the assumption that the set of admissible sequences $\mathcal{B}(A, I)$ is topologically transitive, we say that $\mathcal{B}(A, I)$ admits a {\bf spectral decomposition}, if there are $k \in \mathbb{N}$ and a permutation ${\bf p}$ of the set $\{1, ..., k\}$, such that 
\[
\mathcal{B}(A, I) = \mathcal{B}(A, I)_1 \cup ... \cup \mathcal{B}(A, I)_k \,,
\]
where the subsets $\mathcal{B}(A, I)_i$, with $i \in \{1, ..., k \}$, are pairwise disjoint and closed, $\sigma(\mathcal{B}(A, I)_i) = \mathcal{B}(A, I)_{{\bf p}(i)}$ and each component $\mathcal{B}(A, I)_i$ is topologically mixing for the map $\sigma^{m_i}$, where $m_i$ is the less positive integer such that ${\bf p}^{m_i}(i) = i$. \victor{Note that the transitivity condition guarantees that for each pair $i, j \in \{1, ..., k\}$, there is $n \in \mathbb{N}$ such that $\sigma^n(\mathcal{B}(A, I)_i) = \mathcal{B}(A, I)_j$ which implies that the permutation ${\bf p}$ is necessarily a cycle of length $k$ and, therefore, $m_i = k$ for each $i \in \{1, ..., k\}$..}

\begin{remark}
It is widely known that the any element of the class of finite Markov shifts with irreducible incidence matrix admits a spectral decomposition (see for instance section 1.3 in \cite{MR1484730}). Furthermore, even in non compact approaches this property is  also satisfied. For instance, any countable Markov shift with irreducible incidence matrix satisfies the property. \victor{The above} will be useful in the proof of Proposition \ref{RPF-theorem-SFT} (see for details section 7.1 in \cite{MR1484730}).
\end{remark}

\victor{Throughout the paper we will assume that $M$ is a compact metric space and that the set of admissible sequences $\mathcal{B}(A, I)$ is a topologically transitive set admitting a spectral decomposition. We will suppose also that the map $s$ is locally constant.}

Now we are able to state the main results of this paper. The first one of them is the following:

\begin{theorem}
\label{RPF-theorem}
%\footnote{\rafa{Victor, eu creio que a ultima frase possa ser substituida por: Suppose that $s$ is locally constant. Me parece claro que isso implica em uma condição sobre $A$, que vai se refletir em $s$. Creio que esta redação que proponho é mais simples.}}
For any potential $\varphi \in \mathcal{H}_{\alpha}(\mathcal{B}(A, I))$, the following conditions  are satisfied:
\begin{enumerate}
\item There are $\lambda_{\varphi} > 0$ and a strictly positive function $f_{\varphi} \in \mathcal{H}_{\alpha}(\mathcal{B}(A, I))$, such that, $\mathcal{L}_{\varphi}(f_{\varphi}) = \lambda_{\varphi}f_{\varphi}$. Moreover, the eigenvalue $\lambda_{\varphi}$ is simple and maximal. 
\item There exists a Borelian probability measure $\rho_{\varphi}$ defined on $\mathcal{B}(A, I)$, such that, $\mathcal{L}^*_{\varphi}(\rho_{\varphi}) = \lambda_{\varphi}\rho_{\varphi}$.
\item For $\overline{\varphi} = \varphi + \log(f_{\varphi}) - \log(f_{\varphi} \circ \sigma) - \log(\lambda_{\varphi})$, there is a unique fixed point $\mu_{\varphi}$ for the operator $\mathcal{L}^*_{\overline{\varphi}}$. Moreover, this fixed point is a $\sigma$-invariant probability measure and can be expressed as $d\mu_{\varphi} = f_{\varphi}d\rho_{\varphi}$, with $f_{\varphi}$ satisfying $(1)$ and $\rho_{\varphi}$ satisfying $(2)$.
\item If, in addition, the set of admissible sequences is topologically mixing, then, for any function $\psi \in \mathcal{H}_{\alpha}(\mathcal{B}(A, I))$, we have
\[
\lim_{n \to \infty}\mathcal{L}^n_{\overline{\varphi}}(\psi) = \mu_{\varphi}(\psi) \,, 
\] 
uniformly in the norm $\|\cdot\|_{\infty}$. Moreover, under these assumptions the maximal eigenvalue $\lambda_\varphi$ results isolated as well: the remainder of the spectrum is contained in a disk centered at zero with radius strictly smaller than $\lambda_\varphi$. 
\end{enumerate}
\end{theorem}

%{\bf This part of the extension to the countable Markov shift could appear here...}

Now we propose an application of the former result in the context of countable Markov shifts. In order to do that, consider $M \subset [0, 1]$ a compact set of the form $M := \{b_k : k \in \mathbb{N}\} \cup \{b_{\infty}\}$, where $b_k < b_{k+1}$ for each $k \in \mathbb{N}$. Suppose, also, that $b_{\infty} := 1$ is the unique accumulation point of $M$. It is easy to check that $M$ equipped with the metric $d_M(b_i, b_j) = |b_i - b_j|$ results in a compact metric space.

Thus, choosing $I = \{1\}$, $M_0 := \{b_k : k \in \mathbb{N}\}$ and defining the infinite matrix ${\bf A} \in M_{M_0 \times M_0}(\{0, 1\})$ as ${\bf A}_{b_i, b_j} = 1$, if and only if $A(b_i, b_j) = 1$ and ${\bf A}_{b_i, b_j} = 0$, if and only if $A(b_i, b_j) \neq 1$, we obtain that
\[
\mathcal{B}(A, \{1\}) \cap M_0^{\mathbb{N}} = \Sigma_{\bf A},
\]
where $\Sigma_{\bf A}$ is the {\bf countable Markov shift} on the alphabet $M_0$ with incidence matrix \victor{${\bf A}$, that is,} the set of sequences $x = (x_n)_{n \in \mathbb{N}}$ taking values on the alphabet $M_0$ such that ${\bf A}_{x_n, x_{n+1}} = 1$ for each $n \in \mathbb{N}$. Note that in this case the subshift $\Sigma_{\bf A}$ is a bounded metric space when it is equipped with the metric induced by \eqref{metric}, which in this case is given by
\[
d(x, y) = \sum_{n = 1}^{\infty}\frac{1}{2^n}|x_n - y_n| \,.
\]

Furthermore, since the map $s : M \to \mathcal{K}(M)$ is locally constant, there is $j_0 \in \mathbb{N}$ such that for any $j \geq j_0$ we have ${\bf A}_{b_i, b_j} = {\bf A}_{b_i, b_{j_0}}$ for all $i \in \mathbb{N}$. Note that this class of countable Markov shifts is not immersed in the class of countable Markov shifts satisfying either the BIP property or the finitely primitive condition (see \cite{MR2151222}, \cite{MR2800665} and \cite{MR1955261}).

In fact, we have that ${\bf A}_{b_i, b_j} = ({\bf 1}_I \circ A)(b_i, b_j)$ for any pair $i, j \in \mathbb{N}$, with ${\bf 1}_I$ the characteristic function of the set $I$, that is, the map satisfying ${\bf 1}_I(x) = 1$ if $x \in I$ and ${\bf 1}_I(x) = 0$ if $x \notin I$.

Under these assumptions, we obtain an a priori probability measure given by $\nu = \sum_{k = 1}^{\infty}p_k \delta_{b_k}$, with $\sum_{k = 1}^{\infty}p_k = 1$ and $p_k > 0$ for each $k \in \mathbb{N}$. Note that $b_{\infty} \in \mathrm{supp}(\nu)$, which assures that $\nu$ has full support on the set $M$. Moreover, if we define $p' : M \to [0, 1]$ as $p'(b_k) = p_k$ for each $k \in \mathbb{N}$, $p'(b_{\infty}) = 0$ and $\pi_1 : X \to M$ as $\pi_1(x) = x_1$, we obtain that for each pair $\varphi, \psi \in \mathcal{B}(A, \{1\})$ and any $x \in \Sigma_{\bf A}$, the equation \eqref{Ruelle-operator} can be written as 
\[ 
\mathcal{L}_{\varphi}(\psi)(x) = \sum_{a \in s(x_1)} e^{\varphi(ax)}\psi(ax)(p' \circ \pi_1)(ax) \,. 
\]
In particular, taking $p = p'|_{\Sigma_{\bf A}}$ and using that the map $\psi$ is bounded, it follows that for each $x \in \Sigma_{\bf A}$ we have
\[
\mathcal{L}_{\varphi}(\psi)(x) = L_{\varphi|_{\Sigma_{\bf A}} + \log(p \circ \pi_1)}(\psi|_{\Sigma_{\bf A}}) (x) \,,
\]
where $L_{\phi}$ is the {\bf classical Ruelle operator} associated to $\phi \in \mathcal{C}(\Sigma_{\bf A})$ (which is defined without using an a priori probability measure \victor{but assuming suitable conditions on the potential}), given by the map assigning to each $\xi \in \mathcal{C}(\Sigma_{\bf A})$ the function $L_{\phi}(\xi)$ defined as
\begin{equation}
\label{Ruelle-operator-SFT}
L_{\phi}(\xi)(x) := \sum_{\substack{a \in M_0 \\ {\bf A}_{a, x_1} = 1}} e^{\phi(ax)}\xi(ax) \,,
\end{equation}
for each $x \in \Sigma_{\bf A}$ (see for instance \cite{MR1738951}). Note that the sum  in the right side of \eqref{Ruelle-operator-SFT} could fail to be finite. However,  in Proposition \ref{RPF-theorem-SFT} we will give conditions on the potential $\phi$ in order to guarantee finiteness of that sum for any $\psi \in \mathcal{H}_{\alpha}(\Sigma_{\bf A})$ and thus guarantee that $L_{\phi}$ is well defined in our setting.

Under that assumptions, it follows that \eqref{dual-Ruelle-operator} holds for any Borelian measure $\mu$ defined on $\Sigma_{\bf A}$ and any $\psi \in \mathcal{H}_{\alpha}(\Sigma_{\bf A})$.

A matrix ${\bf A} \in {\bf M}_{M_0 \times M_0}(\{0, 1\})$ is called irreducible, if for any pair $(b_i, b_j)$ belonging to the set $M_0 \times M_0$, there is $k \in \mathbb{N}$ such that the component in the row $b_i$ and the column $b_j$ of the matrix ${\bf A}^k$ is positive. On other hand, a matrix ${\bf A} \in {\bf M}_{M_0 \times M_0}(\{0, 1\})$ is called aperiodic, if there is $k \in \mathbb{N}$ such that all the components of the matrix ${\bf A}^k$ are positive. Note that $\mathcal{B}(A, \{1\})$ is topologically transitive if and only if the matrix ${\bf A}$ is irreducible and the set $\mathcal{B}(A, \{1\})$ is topologically mixing if and only if ${\bf A}$ results in an aperiodic matrix.

\victor{Throughout the paper we will assume that $\Sigma_{\bf A}$ is a countable Markov shift on the alphabet $M_0 = \{b_k : k \in \mathbb{N}\}$ with irreducible incidence matrix ${\bf A}$ and that there is $j_0 \in \mathbb{N}$ such that for any $j \geq j_0$ we have ${\bf A}_{b_i, b_j} = {\bf A}_{b_i, b_{j_0}}$ for all $i \in \mathbb{N}$.

This approach allows} to state a Ruelle's Perron-Frobenius Theorem in the context of countable Markov shifts in the following way:

\begin{proposition}
\label{RPF-theorem-SFT}
\victor{Consider} a potential $\phi : \Sigma_{\bf A} \to \mathbb{R}$ of the form $\phi = \varphi + \log(p \circ \pi_1)$, with $\varphi \in \mathcal{H}_{\alpha}(\Sigma_{\bf A})$ and $p : M_0 \to [0, 1]$ satisfying $p(b_k) > 0$ for each $k \in \mathbb{N}$ and $\sum_{k = 1}^{\infty}p(b_k) = 1$. Then, the Ruelle operator $L_{\phi}$ is well defined and Theorem \ref{RPF-theorem} holds in the following way:
\begin{enumerate}
\item There are $\lambda_{\phi} > 0$ and a strictly positive function $f_{\phi} \in \mathcal{H}_{\alpha}(\Sigma_{\bf A})$ such that $L_{\phi}(f_{\phi}) = \lambda_{\phi}f_{\phi}$. Moreover, the eigenvalue $\lambda_{\phi}$ is simple and maximal.
\item There exists a Borelian probability measure $\rho_{\phi}$ defined on $\Sigma_{\bf A}$ and satisfying $L^*_{\phi}(\rho_{\phi}) = \lambda_{\phi}\rho_{\phi}$.
\item For $\overline{\phi} = \phi + \log(f_{\phi}) - \log(f_{\phi} \circ \sigma) - \log(\lambda_{\phi})$, there is a unique fixed point $\mu_{\phi}$ for the operator $L^*_{\overline{\phi}}$. Moreover, this fixed point is a $\sigma$-invariant probability measure and can be expressed of the form $d\mu_{\phi} = f_{\phi}d\rho_{\phi}$, with $f_{\phi}$ satisfying $(1)$ and $\rho_{\phi}$ satisfying $(2)$.
\item If the incidence matrix ${\bf A}$ is aperiodic, then, for any function $\psi \in \mathcal{H}_{\alpha}(\Sigma_{\bf A})$, we have
\[
\lim_{n \to \infty}L^n_{\overline{\phi}}(\psi) = \mu_{\phi}(\psi) \,, 
\] 
uniformly in the norm $\|\cdot\|_{\infty}$. Furthermore, in this case the eigenvalue $\lambda_\phi$ results isolated. That is, the remainder of the spectrum is contained in a disk centered at zero with radius strictly smaller than $\lambda_\varphi$.
\end{enumerate}
\end{proposition}

On other hand, given a potential $\varphi \in \mathcal{C}(Y)$, with $Y \subset X$ a metric subspace, we say that a probability measure $\mu_{\infty} \in \mathcal{M}_{\sigma}(Y)$ is a {\bf$\varphi$-maximizing measure}, if satisfies 
\[
\mu_{\infty}(\varphi) = m(\varphi) := \sup\{\mu(\varphi) : \mu \in \mathcal{M}_{\sigma}(Y)\} \,.
\]

Hereafter, we will denote by $\mathcal{M}_{\max}(\varphi)$ to the set of all the $\varphi$-maximizing probability measures, which is a non-empty set when $Y \subset X$ is a compact metric space. In section \ref{theory-PF-section} will be proved a variational principle of the pressure for the equilibrium states obtained from Theorem \ref{RPF-theorem}, which implies that the accumulation points of the family of Gibbs states $(\mu_{t\varphi})_{t > 1}$ are in fact $\varphi$-maximizing probability measures. 

\victor{The above allows} to state the following result about existence of maximizing probability measures in the context of countable Markov shifts satisfying the conditions that appear in Proposition \ref{RPF-theorem-SFT}, using techniques of selection and non-selection at zero temperature. Note that this result is stated in an approach that is different to the ones that appear in \cite{MR3864383}, \cite{MR2293630} and \cite{MR2151222}, where, either are assumed another combinatorial conditions on the countable Markov shift $\Sigma_{\bf A}$ or are required another conditions on the regularity of the potential that represents the interactions on the system.

\begin{proposition}
\label{ground-states}
\victor{For each $t > 1$ consider the potential $\phi_t : \Sigma_{\bf A} \to \mathbb{R}$ given by} $\phi_t = t\varphi + \log(p \circ \pi_1)$, with $\varphi \in \mathcal{H}_{\alpha}(\Sigma_{\bf A})$ and $p : M_0 \to [0, 1]$, such that $p(b_k) > 0$ for each $k \in \mathbb{N}$ and $\sum_{k = 1}^{\infty}p(b_k) = 1$. Then, the family of equilibrium states $(\mu_{\phi_t})_{t > 1}$ has an accumulation point $\mu_{\infty}$ at infinity and $\mu_{\infty} \in \mathcal{M}_{\max}(\varphi)$. 
\end{proposition}

Now we move our attention to an interesting setting in thermodynamical formalism: the study of bilateral topological subshifts. The Ruelle operator rely on the fact that the inverse image of any point is composed by several other points, and only can be defined because the shift map is not injective. This is no longer true in the case of bilateral topological subshifts, and therefore  the Ruelle operator can not be defined in these cases. However, the Livsic's Theorem and the use of involution kernels arise as tools to find maximizing measures in these approaches. Below we will show some results in this direction. \rafa{ In particular we will obtain an  expression for the normalized eigenfunction of the Ruelle operator associated to the maximal eigenvalue } \victor{(normalized in the sense that its integral with respect to the eigenprobability is equal to $1$)}, in terms of the eigenprobability of its corresponding dual.

Let $M$ be a compact metric space. Define the set 
\[
\mathcal{B}(A, I)^* := \{(\ldots, y_2, y_1) :\, y_i \in M,\, A(y_{i+1}, y_i) \in I,\, \forall i \in \mathbb{N}\} \,,
\]
with the map $\sigma^* : \mathcal{B}(A, I)^* \to \mathcal{B}(A, I)^*$ given by $\sigma^*((\ldots, y_2, y_1)) = (\ldots, y_3, y_2)$ acting on it. We call $\mathcal{B}(A, I)^*$ the {\bf transpose topological subshift} of $\mathcal{B}(A, I)$. 

Now we can define a bilateral topological subshift associated to $A$ and $I$ through an auxiliary function $\pi_{1, 1} : \mathcal{B}(A, I)^* \times \mathcal{B}(A, I) \to M \times M$ given by the equation $\pi_{1,1}(y, x) = (y_1, x_1)$. Define the set $\widehat{\mathcal{B}(A, I)}$ in the following way:
\[
\widehat{\mathcal{B}(A, I)} := \{(y, x) \in \mathcal{B}(A, I)^* \times \mathcal{B}(A, I) : (A \circ \pi_{1,1})(y, x) \in I\} \,.
\] 

In general, the sets $\widehat{\mathcal{B}(A, I)}$ and $\mathcal{B}(A, I)^* \times \mathcal{B}(A, I)$ don't agree, thus, we will use the notation $(y|x) = (\ldots, y_2, y_1 | x_1, x_2, \ldots)$ for the pairs $(y, x) \in \widehat{\mathcal{B}(A, I)}$. The {\bf bilateral shift map} $\widehat{\sigma} : \mathcal{B}(A, I)^* \times \mathcal{B}(A, I) \to \mathcal{B}(A, I)^* \times \mathcal{B}(A, I)$ is given by $\widehat{\sigma}(y, x) = (\tau^*_x(y), \sigma(x))$, where $\tau^*_x(y) = (\ldots, y_2, y_1, x_1) \in \mathcal{B}(A, I)^*$. \victor{An easy calculation allows} to check that $\widehat{\sigma}$ is invertible, with inverse satisfying the equation $\widehat{\sigma}^{-1}(y, x) = (\sigma^*(y), \tau_y(x))$, where $\tau_y(x) = (y_1, x_1, x_2, \ldots) \in \mathcal{B}(A, I)$. 

It is not difficult to check that $\widehat{\mathcal{B}(A, I)}$ results in a compact $\widehat{\sigma}$-invariant metric space, which implies that it is a bilateral topological subshift whose definition only depends of the function $A$ and the set $I$. Moreover, if the set $\mathcal{B}(A, I)$ is topologically transitive (resp. topologically mixing), we obtain that the sets $\mathcal{B}(A, I)^*$ and $\widehat{\mathcal{B}(A, I)}$ are also topologically transitive (resp. topologically mixing).

Now we will introduce the definition of \victor{involution kernel} associated to a potential $\varphi \in \mathcal{C}(\mathcal{B}(A, I))$. We say that a function $W : \widehat{\mathcal{B}(A, I)} \to \mathbb{R}$ is an {\bf involution kernel} associated to the potential $\varphi \in \mathcal{C}(\mathcal{B}(A, I))$, if the function $\widehat{\varphi} : \widehat{\mathcal{B}(A, I)} \to \mathbb{R}$ defined by $\widehat{\varphi}(y|x) := \varphi(x)$ for any $(y|x) \in \widehat{\mathcal{B}(A, I)}$ and the potential $\widehat{\varphi}^*$ defined by
\begin{equation}
\label{dual-potential}
\widehat{\varphi}^* := \widehat{\varphi} \circ \widehat{\sigma}^{-1} + W \circ \widehat{\sigma}^{-1} - W \,,
\end{equation} 
are such that $\widehat{\varphi}^*(y|x)$ does not depend on $x$, for any 
 $(y|x) \in \widehat{\mathcal{B}(A, I)}$. We will call $\varphi^*(y) :=\widehat{\varphi}^*(y|x)$  %for any $(y|x) \in \widehat{\mathcal{B}(A, I)}$.}
%The potential $\varphi^*$ is called 
the {\bf dual potential} of $\varphi$. \ra{ Some results about the behavior of the involution kernel in the settings of finite Markov shifts or when the alphabet is given by $S^1$}
can be found in \cite{MR2210682, CiLo17}.

%\victor{mudei o ordem deste paragrafo, aparecia um pouco antes do enunciado do Teorema 2, mas é conveniente explicar que o potencial transposto de um Hölde é Hölder antes de definir o operador}
Define $\tau_{y, n}(x) = (y_n, \ldots, y_1, x_1, \ldots)$, fixing $x' \in \mathcal{B}(A, I)$ such that $x'_1 = x_1$, an easy calculation shows that if $\varphi$ is a H\"older continuous function, the map $W : \widehat{\mathcal{B}(A, I)} \to \mathbb{R}$ given by   
\begin{equation}
\label{involution-kernel-example}
W_{\varphi}(y|x) = \sum_{n = 1}^{\infty} \varphi(\tau_{y, n}(x)) - \varphi(\tau_{y, n}(x')) \,,
\end{equation} 
is an involution kernel, $W \in \mathcal{H}_{\alpha}(\widehat{\mathcal{B}(A, I)})$ and $\varphi^* \in \mathcal{H}_{\alpha}(\mathcal{B}(A, I)^*)$.

Assuming that $\varphi \in \mathcal{H}_{\alpha}(\mathcal{B}(A, I))$, by \eqref{involution-kernel-example}, we can consider $\varphi^*$ as a function belonging to $\mathcal{H}_{\alpha}(\mathcal{B}(A, I)^*)$. Therefore, we can define the Ruelle operator associated to $\varphi^*$ as the map that assigns to each $\psi^* \in \mathcal{H}_{\alpha}(\mathcal{B}(A, I)^*)$ the function
\[
\mathcal{L}_{\varphi^*}(\psi^*)(y) := \int_{s^*(y_1)}e^{\varphi^*(ya)}\psi^*(ya) d\nu(a),
\]
where $s^*(b)$ is defined as the set of elements $a \in M$ such that $A(b, a) \in I$ and $ya$ is the concatenation of the sequence $y \in \mathcal{B}(A, I)^*$ and the word $a \in s^*(y_1)$.

An equivalent expression to \eqref{dual-potential} which will be used later is the following: for any $a \in M$, $x$ and $y$, such that, $(y|ax)\in  \widehat{\mathcal{B}(A, I)} $, we have 
\begin{equation}
\label{involution-kernel}
(\widehat{\varphi}^* + W)(ya|x) = (\widehat{\varphi} + W)(y|ax) \,.
\end{equation} 

\rafa{The following Theorem characterizes the normalized eigenfunctions of the Ruelle operators $\mathcal{L}_{\varphi}$ and $\mathcal{L}_{\varphi^*}$ associated to the maximal eigenvalue $\lambda_{\varphi} = \lambda_{\varphi^*}$, in terms of the involution kernel and the eigenprobabilities $\rho_{\varphi}$ and $\rho_{\varphi^*}$, given by Theorem \ref{RPF-theorem}. }
%(\rafa{The existence of such eigenfunctions was assured in Theorem \ref{RPF-theorem}.})
 In particular, this result works for the map that appears in \eqref{involution-kernel-example}.

\begin{theorem}
\label{bilateral-extension}
\victor{Assume that the potentials $\varphi \in \mathcal{H}_{\alpha}(\mathcal{B}(A, I))$, $W \in \mathcal{H}_{\alpha}(\widehat{\mathcal{B}(A, I)})$ and $\varphi^* \in \mathcal{H}_{\alpha}(\mathcal{B}(A, I)^*)$ satisfy} \eqref{dual-potential}. Set 
\[
c := \log\bigl((\rho_{\varphi^*} \times \rho_{\varphi})(({\bf 1}_I \circ A \circ \pi_{1,1})e^W)\bigr) \,.
\]
Then: 
\begin{enumerate} 
\item %There are explicit expressions for the eigenfunctions associated to the operators $\mathcal{L}_{\varphi}$ and $\mathcal{L}_{\varphi^*}$ in terms of the eigenprobabilities $\rho_{\varphi}$ and $\rho_{\varphi^*}$, given by Theorem \ref{RPF-theorem}: 
\rafa {If we define 
\[
f = \rho_{\varphi^*}\bigl(({\bf 1}_I \circ A \circ \pi_{1, 1})(y, \cdot)e^{W(y|\cdot) - c}\bigr) \,,
\]
\[
f^* = \rho_{\varphi}\bigl(({\bf 1}_I \circ A \circ \pi_{1, 1})(\cdot, x)e^{W(\cdot|x) - c}\bigr) \,.
\]
Then,} $\mathcal{L}_{\varphi}(f) = \lambda_{\varphi}f$, $\mathcal{L}_{\varphi^*}(f^*) = \lambda_{\varphi^*}f^*$, $\rho_{\varphi}(f) = 1 = \rho_{\varphi^*}(f^*)$ and $\lambda_{\varphi} = \lambda_{\varphi^*}$.
\item Let $\mu_{\varphi}$ be the Gibbs state associated to the potential $\varphi$, given by item (3) of Theorem \ref{RPF-theorem}. There is a natural extension of $\mu_{\varphi}$ to a Borelian measure $\mu_{\widehat{\varphi}}$ on the set $\widehat{\mathcal{B}(A, I)}$, which is given by
\[
d\mu_{\widehat{\varphi}} := ({\bf 1}_I \circ A \circ \pi_{1, 1})e^{W - c}d(\rho_{\varphi^*} \times \rho_{\varphi}) \,.
\]

By a natural extension we mean that, for any $\psi \in \mathcal{C}(\mathcal{B}(A, I))$, the potential $\widehat{\psi} \in \mathcal{C}(\widehat{\mathcal{B}(A, I)})$ defined by $\widehat{\psi}(y|x) = \psi(x)$ for each $(y|x) \in \widehat{\mathcal{B}(A, I)}$ satisfies
\[
\mu_{\widehat{\varphi}}(\widehat{\psi}) = \mu_{\varphi}(\psi) \,,
\]
and for any $\psi^* \in \mathcal{C}(\mathcal{B}(A, I)^*)$, the potential $\widehat{\psi} \in \mathcal{C}(\widehat{\mathcal{B}(A, I)})$ defined by $\widehat{\psi}(y|x) = \psi^*(y)$ for each $(y|x) \in \widehat{\mathcal{B}(A, I)}$ satisfies
\[
\mu_{\widehat{\varphi}}(\widehat{\psi}) = \mu_{\varphi^*}(\psi^*) \,.
\]

Furthermore, any accumulation point $\widehat{\mu}_\infty$ of the family $(\mu_{t \widehat{\varphi}})_{t > 1}$ at infinity is maximizing for the potential $\widehat{\varphi}$ in the following way
\[
\widehat{\mu}_\infty(\widehat{\varphi}) = m(\varphi) \,.
\]
\end{enumerate}
\end{theorem}

%\rafa{In the setting of potentials depending on two coordinates (that is, when $f(x) = f(x_1, x_2)$ for each $x \in \mathcal{B}(A, I)$), the expressions stated in the above Theorem allow to find eigenprobabilities and Gibbs states via a Markovian transition kernel. Furthermore, in that case  Gibbs states can be characterized as stationary Markov measures,  the so called Markovian transition kernel depends on the involution kernel $W$, and regularity conditions on the involution kernel (for example differentiability on each one of the coordinates), imply the same regularity on the eigenfunction associated to the potential (see for instance \cite{MR2864625}, \cite{MR3377291} and \cite{MR2496111}). 

\victor{We observe that regularity conditions on the involution kernel (for example differentiability on each one of the coordinates), imply the same regularity on the eigenfunction associated to the potential (see \cite{MR3377291}). In some cases, for instance, when the potential $\varphi$ is a function of product type on a finite alphabet and all the sequences taking values into the alphabet are allowed (i.e. in the case of finite full shifts), it follows that the eigenprobability $\rho_\varphi$ and the eigenfunction $f_\varphi$ have an explicit form in terms of the potential $\varphi$.

Indeed, fixing $M = \{-1, 1\}$, $A(i, j) = 1$ for $i, j \in \{-1, 1\}$, and a continuous potential of the form $\varphi(x) = \sum_{n=1}^\infty a_n x_n$, it follows that the involution kernel is of the form $W_\varphi(x) = \sum_{n=1}^\infty (x_n + y_n)\sum_{i = n+1}^\infty a_i$. Furthermore, for any cylinder $[a_1, ... , a_n]$, we have $\rho_\varphi([a_1, ... , a_n]) = \prod_{k=1}^n \mu_k (a_k)$, where each measure $\mu_k$ on $M$ is given by $\mu_k(a) := \Bigl( \prod_{i=1}^k e^{-a_i - a_i a} + \prod_{i=1}^k e^{a_i - a_i a}\Bigr)^{-1}$ for $a \in \{-1, 1\}$. Besides that, we have $f_\varphi(x) = \prod_{n=1}^\infty f_n(x_n)$, where $f_n(a) = \prod_{k = n+1}^\infty e^{a_k a}$ for $a \in \{-1, 1\}$, and $\lambda_\varphi = \prod_{n=1}^\infty e^{-a_n} + \prod_{n=1}^\infty e^{a_n}$ (see for details \cite{MR3656287}, \cite{CiLo17} and \cite{Moh18}).}

\section{Theory of Perron-Frobenius}
\label{theory-PF-section}

The theory of Perron-Frobenius is a useful tool to find vector subspaces that remain invariant by the action of a linear operator, which, in the context of thermodynamical formalism, arises as a way to find probability measures that optimize the energy of a system modeled on a topological subshift with interactions described by a potential, the above through eigenvalues and eigenvectors associated to a transfer operator and its corresponding dual. In this section we will prove  Theorem \ref{RPF-theorem} and Proposition \ref{RPF-theorem-SFT}. We will also prove a variational principle in order to show that the Gibbs states from Theorem \ref{RPF-theorem} and Proposition \ref{RPF-theorem-SFT} are in fact equilibrium states.

\rafa{We now present the proof of Theorem \ref{RPF-theorem}:}

\begin{proof}[Proof of Theorem \ref{RPF-theorem}]
\label{proof-RPF-theorem}
Define $\mathcal{T}_{t, \varphi}$ as the operator assigning to each $u \in \mathcal{C}(\mathcal{B}(A, I))$ the function $\mathcal{T}_{t, \varphi}(u) = \log\bigl( \mathcal{L}_{\varphi}(e^{tu}) \bigr)$. Since the Ruelle operator preserves the set of continuous functions, it follows that $\mathcal{T}_{t, \varphi}(u) \in \mathcal{C}(\mathcal{B}(A, I))$. We begin by proving that, for each $t \in (0, 1)$, the operator $\mathcal{T}_{t, \varphi}$ is a uniform contraction. 

Indeed, for any pair $u, v \in \mathcal{C}(\mathcal{B}(A, I))$ we have
 %Moreover, for each $t \in (0, 1)$, the operator $\mathcal{T}_{t, \varphi}$ is a uniform contraction. Indeed, for any pair $u_1, u_2 \in \mathcal{C}(\mathcal{B}(A, I))$ is satisfied

\victor{
\begin{equation*}
\|\mathcal{T}_{t, \varphi}(u) - \mathcal{T}_{t, \varphi}(v)\|_{\infty}
%= \Bigl\| \log\Bigl( \frac{\mathcal{L}_{\varphi}(e^{tu})}{\mathcal{L}_{\varphi}(e^{tv})} \Bigr) \Bigr\|_{\infty} 
\leq \Bigl\| \log\Bigl( \frac{e^{t\|u - v\|_{\infty}}\mathcal{L}_{\varphi}(e^{tv})}{\mathcal{L}_{\varphi}(e^{tv})} \Bigr) \Bigr\|_{\infty} 
= t\|u - v\|_{\infty} \nonumber \,.
\end{equation*}
}

By the above, as a consequence of the Banach's Fixed Point Theorem, it follows that for each $t \in (0, 1)$ there is a function $u_t \in \mathcal{C}(\mathcal{B}(A, I))$ such that $\mathcal{T}_{t, \varphi}(u_t) = u_t$, that is, $e^{u_t} = \mathcal{L}_{\varphi}(e^{tu_t})$. 

Now, we will check that the family $(u_t)_{0<t<1}$ is equicontinuous. Since $s$ is locally constant, for each $z \in \mathcal{B}(A, I)$, there is $\epsilon_z > 0$ such that for any $y \in \mathcal{B}(A, I)$ such that $y_1 \in (z_1 - \epsilon_z, z_1 + \epsilon_z)$ we have $s(y_1) = s(z_1)$. 
Denote by 
\[
V_z = \{y \in \mathcal{B}(A, I) : y_1 \in (z_1 - \epsilon_z, z_1 + \epsilon_z) \} \,.
\]

Then, $V_z$ is an open neighborhood of $z$ in $\mathcal{B}(A, I)$, 
%\rafa{and $\cup_{z\in \mathcal{B}(A, I)}V_z$ is an open cover of the compact set $\mathcal{B}(A, I)$. Let $\delta$ be a Lebesgue number of this covering.}
and for any pair of points $x, y \in V_z$, %such that $d(x,y)<\delta$, 
we have $s(x_1) = s(y_1) = s(z_1)$. Thus, for each $t \in (0, 1)$ and any pair $x, y \in V_z$, we have
\victor{
 \begin{equation*}
e^{u_t(x)} 
= \mathcal{L}_{\varphi}(e^{tu_t})(x) 
%&\leq \sup\Bigl\{e^{\varphi(ax) - \varphi(ay) + tu_t(ax) - tu_t(ay)} : a \in s(z_1) \Bigr\} \mathcal{L}_{\varphi}(e^{tu_t})(y) \nonumber \\
\leq \sup\Bigl\{e^{\varphi(ax) - \varphi(ay) + tu_t(ax) - tu_t(ay)} : a \in s(z_1) \Bigr\} e^{u_t(y)} \,.
\end{equation*}
}

The above implies that 
\[
|u_t(x) - u_t(y)| \leq \sup\{\varphi(ax) - \varphi(ay) + tu_t(ax) - tu_t(ay) : a \in s(z_1)\} \,.
\]

Moreover, if we use the notation $a_0 := z_1$, following an inductive argument, it is easy to check that for any $n \in \mathbb{N}$, each $a^n = a_n \ldots a_1$, and any pair $x, y \in V_z$, we have 
\begin{align}
&|u_t(x) - u_t(y)| \nonumber \\ 
&\leq \sup\Bigl\{\sum_{j=1}^n t^{j-1}(\varphi(a^jx) - \varphi(a^jy)) + t^n(u_t(a^nx) - u_t(a^ny)) : a_{j+1} \in s(a_j) \Bigr\} \nonumber \\
&\leq \sum_{j=1}^n \frac{t^{j-1}}{2^{\alpha j}}\mathrm{Hol}_{\varphi}d(x, y)^{\alpha} + 2t^n\|u_t\|_{\infty} \nonumber \,.
\end{align} 
%\rafa{(Note that the distance $d(a^jx,a^jy)<d(x,y)$ is bounded by the Lebesgue number $\delta$, and therefore $s(a^jx)=s(a^jy)$ for each $a^j=a_j \ldots a1$, which shows we can use the inductive argument above.)}

Then, taking the limit when $n \to \infty$ in the right side of the last inequality, it follows that 
\victor{
\begin{equation*} 
|u_t(x) - u_t(y)| 
\leq \sum_{j = 1}^{\infty} \frac{t^{j-1}}{2^{\alpha j}}\mathrm{Hol}_{\varphi}d(x, y)^{\alpha} \nonumber \\ 
%= \frac{1}{2^{\alpha} - t}\mathrm{Hol}_{\varphi}d(x, y)^{\alpha} 
< \frac{1}{2^{\alpha} - 1}\mathrm{Hol}_{\varphi}d(x, y)^{\alpha} 
%\leq \mathrm{Hol}_{\varphi}d(x, y)^{\alpha} 
\,.
\end{equation*}
}

By the above, the function $u_t|_{V_z}$ is H\"older continuous. Furthermore, denoting by $\mathrm{Hol}_{t, z}$ the corresponding H\"older \victor{constant of $u_t|_{V_z}$,} we have $\mathrm{Hol}_{t, z} \leq \frac{1}{2^{\alpha} - 1}\mathrm{Hol}_{\varphi}$ and thus $|u_t(x) - u_t(y)| \leq \frac{1}{2^{\alpha} - 1}\mathrm{Hol}_{\varphi} d(x,y)^{\alpha}$ for any pair of points $x, y \in V_z$.

Since $\mathcal{B}(A, I)$ is a compact set and  $\mathcal{B}(A, I) \subset \cup_{z \in \mathcal{B}(A, I)} V_z$, there is a finite collection of points $\{z^1, \ldots, z^n\}$ such that $\mathcal{B}(A, I) \subset \cup_{i = 1}^n V_{z^i}$, which implies that $u_t \in \mathcal{H}_{\alpha}(\mathcal{B}(A, I))$, with $\mathrm{Hol}_{u_t} \leq c\mathrm{Hol}_{\varphi}$ for some constant $c > 0$ that depends only on the collection $\{z^1, \ldots, z^n\}$. 

Therefore, for any $t \in (0, 1)$ and each $x, y \in \mathcal{B}(A, I)$ we have
\begin{equation}
\label{equicontinuity} 
|u_t(x) - u_t(y)| \leq c\mathrm{Hol}_{\varphi} d(x, y)^{\alpha} \,.
\end{equation} 
As a consequence, the family $(u_t)_{0 < t < 1}$ is equicontinuous, as we wanted to prove. 

Now we define $u_t^* = u_t - \max(u_t)$. The family  $(u_t^*)_{0<t<1}$ is: (a) equicontinuous and (b) uniformly bounded \victor{(see for instance \cite{MR2864625} and \cite{MR3377291}). 
Furthermore, taking $u = \lim_{n \to \infty} u_{t_n}^*$ and $\kappa = \lim_{n \to \infty}(1 - t_n)\max(u_{t_n})$, where $(t_n)_{n \in \mathbb{N}}$ is a suitable sequence (see for details \cite{MR2864625}), we obtain that $\mathcal{L}_{\varphi}(e^{u}) = e^{\kappa}e^{u}$.
%where the last  equality is a consequence of the Dominated Convergence Theorem. 
%Using the fact that, for any $x, y \in \mathcal{B}(A, I)$ we have
%\[
%|u(x) - u(y)| \leq c\mathrm{Hol}_{\varphi} d(x, y)^{\alpha} \,,
%\]
Besides that, $u$ results in a H\"older continuous function which implies} that $e^{u}$  belongs to the set $\mathcal{H}_{\alpha}(\mathcal{B}(A, I))$ and, also, is strictly positive. Hereafter, we will use the notation $f_{\varphi} = e^{u}$ and $\lambda_{\varphi} = e^{\kappa}$.

Now we will check that the eigenvalue $\lambda_{\varphi}$ is simple. In order to do that, assume that $f_1$ is another eigenfunction associated to the eigenvalue $\lambda_{\varphi}$. Set $\widetilde{t} = \min\Bigl\{\frac{f_1}{f_{\varphi}}\Bigr\}$, which is well defined because the function $f_{\varphi}$ is a strictly positive continuous function defined on a compact set. Moreover, by continuity of the function $\frac{f_1}{f_{\varphi}}$ and compactness of the set $\mathcal{B}(A, I)$, there is $\widetilde{x} \in \mathcal{B}(A, I)$ such that $\widetilde{t} = \frac{f_1(\widetilde{x})}{f_{\varphi}(\widetilde{x})}$. Thus, $f_2 \equiv f_1 - \widetilde{t} f_{\varphi}$ is a non-negative continuous function that attains its minimum value at $0$ in the point $\widetilde{x}$, which implies that
\victor{
\[
0 
= \lambda_{\varphi}^n f_2(\widetilde{x}) 
%= \mathcal{L}_{\varphi}^n f_2(\widetilde{x}) 
= \int_{s(a_{n-1})} \ldots \int_{s(\widetilde{x}_1)}e^{S_n \varphi(a^n\widetilde{x})}f_2(a^n\widetilde{x}) d\nu(a_1) \ldots d\nu(a_n) \,.
\]
}

In particular, since $\nu$ has full support and $e^{S_n\varphi}, f_2$ are non-negative continuous functions, we obtain that $f_2(a^n\widetilde{x}) = 0$ for each word $a^n = a_n \ldots a_1$ such that $a_1 \in s(\widetilde{x}_1), a_2 \in s(a_1), \ldots, a_n \in s(a_{n-1})$. Now, it follows from the transitivity of the set of admissible sequences $\mathcal{B}(A, I)$ that the set $\cup_{n=0}^{\infty}\sigma^{-n}(\{\widetilde{x}\})$ is a dense subset of $\mathcal{B}(A, I)$ (see for instance section $4.2$ in \cite{MR1450400}), which implies that $f_2 \equiv 0$ as a consequence of the continuity. Therefore, we obtain that the eigenvalue $\lambda_{\varphi}$ is simple.

To finish the proof of item (1) of Theorem \ref{RPF-theorem}, we still have to
 prove that $\lambda_{\varphi}$ is a maximal eigenvalue for the operator $\mathcal{L}_{\varphi}$. 

Before that, we need to prove the other items of the Theorem. We begin by item (2) of Theorem \ref{RPF-theorem}: Define $\overline{\mathcal{L}_{\varphi}^*}$ as the operator assigning to each Borelian measure $\mu$ on $\mathcal{B}(A, I)$, the Borelian measure given by $\overline{\mathcal{L}_{\varphi}^*}(\mu) = \frac{1}{\mathcal{L}_{\varphi}^*(\mu)(1)}\mathcal{L}_{\varphi}^*(\mu)$. \victor{By a straightforward argument (see for instance \cite{MR0234697}), we can guarantee existence of a Borelian probability measure $\rho_{\varphi}$ such that}
%We have  $\overline{\mathcal{L}_{\varphi}^*}(\mu)(1) = 1$ which implies  this operator preserves the set of Borelian probability measures. Then, it follows from Schauder-Tychonoff's Theorem that there is a Borelian probability measure $\rho_{\varphi}$ such that 
\begin{equation}
\label{dual-fixed-point}
\overline{\mathcal{L}_{\varphi}^*}(\rho_{\varphi}) = \rho_{\varphi} \,. 
\end{equation}

%In particular 
%\begin{align}
%\rho_{\varphi}(f_{\varphi}) = \overline{\mathcal{L}_{\varphi}^*}(\rho_{\varphi})(f_{\varphi}) 
%&= \frac{1}{\mathcal{L}_{\varphi}^*(\rho_{\varphi})(1)}\mathcal{L}_{\varphi}^*(\rho_{\varphi})(f_{\varphi}) \nonumber \\
%&= \frac{1}{\mathcal{L}_{\varphi}^*(\rho_{\varphi})(1)}\rho_{\varphi}(\mathcal{L}_{\varphi}(f_{\varphi})) = \frac{\lambda_{\varphi}}{\mathcal{L}_{\varphi}^*(\rho_{\varphi})(1)}\rho_{\varphi}(f_{\varphi}) \nonumber \,. 
%\end{align} 

\victor{The above implies $\mathcal{L}_{\varphi}^*(\rho_{\varphi})(1) = \lambda_{\varphi}$.} Thus, by \eqref{dual-fixed-point}, it follows that $\mathcal{L}_{\varphi}^*(\rho_{\varphi}) = \lambda_{\varphi} \rho_{\varphi}$, which concludes the proof of item (2) of Theorem \ref{RPF-theorem}.

Define $d\mu_{\varphi} = f_{\varphi}d\rho_{\varphi}$, which we will assume w.l.o.g. a probability measure (choosing a suitable eigenfunction $f_{\varphi}$ in such a way that $\rho_{\varphi}(f_{\varphi}) = 1$, \victor{ that is, the so called normalized one). It is not difficult to check that $\mu_{\varphi}$ is a fixed point for the operator $\mathcal{L}_{\overline{\varphi}}^*$ 
%Indeed, by definition of $\overline{\varphi}$, for any $\psi \in \mathcal{H}_{\alpha}(\mathcal{B}(A, I))$, we have
%\begin{align}
%\mathcal{L}_{\overline{\varphi}}^*(\mu_{\varphi})(\psi) = \mu_{\varphi}(\mathcal{L}_{\overline{\varphi}}(\psi)) 
%&= \frac{1}{\lambda_{\varphi}} \rho_{\varphi}(\mathcal{L}_{\varphi}(\psi f_{\varphi})) \nonumber \\
%&= \frac{1}{\lambda_{\varphi}}\mathcal{L}_{\varphi}^*(\rho_{\varphi})(\psi f_{\varphi}) = \rho_{\varphi}(\psi f_{\varphi})(\psi) = \mu_{\varphi}(\psi) \label{fixed-point} \,.
%\end{align}
%Besides that, it follows by the above that for any $\psi \in \mathcal{H}_{\alpha}(\mathcal{B}(A, I))$ is satisfied
%\begin{equation}
%\label{invariant-Gibbs-state}
%\mu_{\varphi}(\psi \circ \sigma) = \mathcal{L}_{\overline{\varphi}}^*(\mu_{\varphi})(\psi \circ \sigma) = \mu_{\varphi}(\mathcal{L}_{\overline{\varphi}}(\psi \circ \sigma)) = \mu_{\varphi}(\psi) \,.
%\end{equation}
and the above implies that the probability measure $\mu_{\varphi}$ is $\sigma$-invariant, which concludes the proof of item (3) of Theorem \ref{RPF-theorem}.

%\vspace{0.5 cm}
%
%\rafa{ATÉ AQUI TUDO PARECE OK NESTA DEMONSTRACAO }
%
%\vspace{0.5 cm}

In order to prove the item (4) of Theorem \ref{RPF-theorem}, first note} that for any pair $\varphi, \psi \in \mathcal{H}_{\alpha}(\mathcal{B}(A, I))$ we have
\[
\Bigl|\mathcal{L}_{\overline{\varphi}}(\psi)(x) - \mathcal{L}_{\overline{\varphi}}(\psi)(y)\Bigr| \leq \frac{1}{2^{\alpha}}\Bigl(\mathrm{Hol}_{e^{\overline{\varphi}}}\|\psi\|_{\infty} + \mathrm{Hol}_{\psi}\Bigr)d(x, y)^{\alpha} \,.
\]

\victor{Then, it follows from an inductive argument that
%\[
%\Bigl|\mathcal{L}^n_{\overline{\varphi}}(\psi)(x) - \mathcal{L}^n_{\overline{\varphi}}(\psi)(y)\Bigr| \leq \Bigl(\mathrm{Hol}_{e^{\overline{\varphi}}}\|\psi\|_{\infty}\Bigl(\sum^n_{j=1}\frac{1}{2^{j\alpha}}\Bigr) + \frac{\mathrm{Hol}_{\psi}}{2^{n\alpha}}\Bigr)d(x, y)^{\alpha} \,,
%\]
%which implies
%\victor{a desigualdade é valida, mas o argumento do limite está errado}
\begin{equation}
\label{preserves-Holder}
\Bigl|\mathcal{L}^n_{\overline{\varphi}}(\psi)(x) - \mathcal{L}^n_{\overline{\varphi}}(\psi)(y)\Bigr| \leq \frac{2^{\alpha}}{2^{\alpha} - 1}\Bigl(\mathrm{Hol}_{e^{\overline{\varphi}}}\|\psi\|_{\infty} + \mathrm{Hol}_{\psi}\Bigr)d(x, y)^{\alpha} \,.
\end{equation}

The last} inequality means the sequence $(\mathcal{L}^n_{\overline{\varphi}}(\psi))_{n \in \mathbb{N}}$ is equicontinuous. Besides that, \eqref{preserves-Holder} guarantees that the operator $\mathcal{L}_{\overline{\varphi}}$ preserves the set $\mathcal{H}_\alpha(\mathcal{B}(A, I))$. Furthermore, since $\mathcal{L}^n_{\overline{\varphi}}(1) = 1$ for each $n \in \mathbb{N}$, it follows that $\|\mathcal{L}^n_{\overline{\varphi}}(\psi)\|_{\alpha} \leq \|\psi\|_{\alpha}$ for each $n \in \mathbb{N}$, which also implies that the sequence $(\mathcal{L}^n_{\overline{\varphi}}(\psi))_{n \in \mathbb{N}}$ is uniformly bounded with the norm $\| \cdot \|_\alpha$. \victor{Therefore, as a consequence of the Arzela-Ascoli's Theorem (see for instance \cite{MR1085356}), we obtain that for all $n \in \mathbb{N}$ we have} 
%Therefore, as a consequence of the Arzela-Ascoli's Theorem, we can guarantee existence of a convergent subsequence $(\mathcal{L}^{n_k}_{\overline{\varphi}}(\psi))_{k \in \mathbb{N}}$ in the norm $\|\cdot\|_{\infty}$. Moreover, since $\mathcal{L}^{n_k}_{\overline{\varphi}}(\psi) \in \mathcal{H}_{\alpha}(\mathcal{B}(A, I))$ for each $k \in \mathbb{N}$, it follows immediately that the function $\widetilde{\psi} = \lim_{k \to \infty} \mathcal{L}^{n_k}_{\overline{\varphi}}(\psi)$ belongs to $\mathcal{H}_{\alpha}(\mathcal{B}(A, I))$ as well. Also we have that for each $n \in \mathbb{N}$ is satisfied 
%\[
%\sup\{\mathcal{L}^{n+1}_{\overline{\varphi}}({\psi})(x) : x \in \mathcal{B}(A, I)\} \leq \sup\{\mathcal{L}^n_{\overline{\varphi}}({\psi})(x) : x \in \mathcal{B}(A, I) \} \,,
%\]
%which implies that we have, for all $n \in \mathbb{N}$, 
\[
\sup\{\widetilde{\psi}(x) : x \in \mathcal{B}(A, I)\} = \sup\{\mathcal{L}^n_{\overline{\varphi}}(\widetilde{\psi})(x) : x \in \mathcal{B}(A, I) \} \,.
\] 

Define as $\widetilde{\psi}_0 =  \sup\{\widetilde{\psi}(x) : x \in \mathcal{B}(A, I)\}$. Thus, we can choose a collection $\{x^n : n \in \mathbb{N} \cup \{0\}\}$ such that for all $n \in \mathbb{N}$, 
\[
\mathcal{L}^n_{\overline{\varphi}}(\widetilde{\psi})(x^n) =  \widetilde{\psi}_0 \,.
\]

The above implies that for each $n \in \mathbb{N}$,  \victor{ 
\[
0 
%= \mathcal{L}^n_{\overline{\varphi}}(\widetilde{\psi}_0 - \widetilde{\psi})(x^n) 
= \int_{s(a_{n-1})} \ldots \int_{s(x_1)}e^{S_n\overline{\varphi}(a^nx^n)}(\widetilde{\psi}_0 - \widetilde{\psi}(a^nx^n)) d\nu(a_1)\ldots d\nu(a_n) \,.
\]}

Then, since the maps $e^{S_n\overline{\varphi}}, \widetilde{\psi}_0 - \widetilde{\psi}$ are non-negative continuous functions and the a priori probability measure $\nu$ has full support, it follows that $\widetilde{\psi}(a^nx^n) = \widetilde{\psi}_0$ for each word $a^n = a_1 \ldots a_n$ such that $a_1 \in s(x^n_1), a_2 \in s(a_1), \ldots, a_n \in s(a_{n-1})$. Since the set of admissible sequences $\mathcal{B}(A, I)$ is topologically mixing, it follows that the set $\cup_{n=1}^{\infty} \sigma^{-n}(\{x^n\})$ is dense in $\mathcal{B}(A, I)$ (see for instance section $4.2$ in \cite{MR1450400} and chapter $2$ in \cite{MR1085356}), thus, it follows that $\widetilde{\psi} \equiv \widetilde{\psi}_0$. The above implies that \victor{
\[
\widetilde{\psi} 
= \mu_{\varphi}(\widetilde{\psi}) 
%= \lim_{k \to \infty}\mu_{\varphi}(\mathcal{L}^{n_k}_{\overline{\varphi}}(\psi)) 
= \lim_{k \to \infty}\mathcal{L}^{*,n_k}_{\overline{\varphi}}(\mu_{\varphi})(\psi) = \mu_{\varphi}(\psi)\,,  
\] }
where the second equality is a consequence of the Dominated Convergence Theorem. Note that the last  equality guarantees that $\widetilde{\psi}$ is independent of the sequence $(n_k)_{k \in \mathbb{N}}$. That is, $\widetilde{\psi}$ is the unique accumulation point of the sequence $(\mathcal{L}^n_{\overline{\varphi}}(\psi))_{n \in \mathbb{N}}$, which implies that
\begin{equation}
\label{limit-spectral-gap}
\lim_{n \to \infty}\mathcal{L}^n_{\overline{\varphi}}(\psi) = \mu_{\varphi}(\psi) \,,
\end{equation}
uniformly in the norm $\|\cdot\|_{\infty}$, as we wanted to prove. 

Now, in order to finish this proof, we just need to prove that $\lambda_{\varphi}$ is a maximal and isolated eigenvalue for the operator $\mathcal{L}_{\varphi}$ when the set of admissible sequences $\mathcal{B}(A, I)$ is topologically mixing. 

First note that replacing $\overline{\varphi}$ by $\varphi$ in \eqref{preserves-Holder}, we obtain that the operator $\mathcal{L}_\varphi$ preserves the set $\mathcal{H}_\alpha(\mathcal{B}(A, I))$. Besides that, as a consequence of \eqref{limit-spectral-gap}, we have that for each $\psi \in \mathcal{H}_{\alpha}(\mathcal{B}(A, I))$ is satisfied
\begin{equation}
\label{spectral-gap} 
\lim_{n \to \infty}\lambda_{\varphi}^{-n}\mathcal{L}_{\varphi}^n(\psi) = f_{\varphi}\rho_{\varphi}(\psi) \;,
\end{equation}
uniformly in the norm $\|\cdot\|_{\infty}$.

\victor{Consider the set $V := \{\psi \in \mathcal{H}_{\alpha}(\mathcal{B}(A, I)) : \rho_{\varphi}(\psi) = 0\}$. 
%\victor{Since $f_\varphi$ is a strictly positive function and $d\mu_\varphi = f_\varphi d\rho_\varphi$, we have that $|\mu_\varphi(\psi)| \leq \|f_\varphi\|_\infty |\rho_\varphi(\psi)|$ and $|\rho_\varphi(\psi)| \leq \|1/f_\varphi\|_\infty |\mu_\varphi(\psi)|$. That is, $\rho_{\varphi}(\psi) = 0$ if, and only if, $\mu_{\varphi}(\psi) = 0$.} 
%Hence, the set $V$ also can be expressed as $V = \{\psi \in \mathcal{H}_{\alpha}(\mathcal{B}(A, I)) : \mu_{\varphi}(\psi) = 0\}$. 
Note that $V$ is a closed vector subspace} of $\mathcal{H}_{\alpha}(\mathcal{B}(A, I))$ and, by \eqref{spectral-gap}, we have that $\mathrm{span}\{f_\varphi\} \cap V = 0$. Therefore, $\mathcal{H}_{\alpha}(\mathcal{B}(A, I)) = \mathrm{span}\{f_\varphi\} \oplus V$. Besides that, for any $\psi \in V$ we have 
\[
\rho_\varphi(\mathcal{L}_\varphi(\psi)) = \mathcal{L}_\varphi^*(\rho_\varphi)(\psi) = \lambda_\varphi \rho_\varphi(\psi) = 0 \;,
\]
which implies that the subspace $V$ is invariant by the action of $\mathcal{L}_\varphi$. Now, let us assume that $\psi_\lambda \in V$ is an eigenfunction of the operator $\mathcal{L}_{\varphi}$ associated to an eigenvalue $\lambda$. Note first that $\lambda$ is necessarily different of $\lambda_{\varphi}$ because $\rho_{\varphi}(\psi_\lambda) = 0$. By \eqref{spectral-gap}, it follows that 
\[
\lim_{n \to \infty}\lambda_{\varphi}^{-n}|\lambda|^n\|\psi_\lambda\|_{\infty} = 0 \,,
\]
which implies that $|\lambda| < \lambda_{\varphi}$ and proves the maximality of the eigenvalue $\lambda_\varphi$. 

In order to prove that $\lambda_\varphi$ is an isolated eigenvalue of $\mathcal{L}_\varphi$, we consider the set $S_V := \{\psi \in V :\; \| \psi \|_\infty = 1,\; \mathrm{Hol}_\psi \leq 1\}$. First note that for any $\psi \in S_V$ we have
\[
|\psi(x) - \psi(y)| \leq \mathrm{Hol}_\psi d(x, y)^\alpha \leq d(x, y)^\alpha \;.
\]

Therefore, the set $S_V$ is equicontinuous and it is uniformly bounded with the norm $\| \cdot \|_\alpha$. By the above, it follows from the Arzela-Ascoli's Theorem that the set $S_V$ is sequentially compact in $\mathcal{C}(\mathcal{B}(A, I))$. In particular, the set $S_V$ results in a compact subset of $\mathcal{C}(\mathcal{B}(A, I))$. On other hand, for each $\psi \in S_V$ the sequence $(\mathcal{L}^n_{\overline{\varphi}}(\psi))_{n \in \mathbb{N}}$ converges to $0$ uniformly in the norm $\|\cdot\|_{\infty}$ and satisfies the following inequalities
\[
\|\mathcal{L}^{n+1}_{\overline{\varphi}}(\psi)\|_\infty = \|\mathcal{L}_{\overline{\varphi}}(\mathcal{L}^n_{\overline{\varphi}}(\psi))\|_\infty \leq \|\mathcal{L}_{\overline{\varphi}}\| \|\mathcal{L}^n_{\overline{\varphi}}(\psi)\|_\infty \leq \|\mathcal{L}^n_{\overline{\varphi}}(\psi)\|_\infty \;.
\]

That is, for each $\psi \in S_V$ the sequence $(\mathcal{L}^n_{\overline{\varphi}}(\psi))_{n \in \mathbb{N}}$ is decreasing. Therefore, it follows from the Dini's Theorem that sequence $(\mathcal{L}^n_{\overline{\varphi}})_{n \in \mathbb{N}}$ is uniformly convergent to $0$ on the set $S_V$ with the operator norm. Furthermore, the above implies that the sequence $(\lambda^{-n}\mathcal{L}^n_\varphi)_{n \in \mathbb{N}}$ is also uniformly convergent to $0$ on the set $S_V$ with the operator norm. 

Then, taking $0 <  r_0 < 1$, also by \eqref{spectral-gap}, there is $n_0 \in \mathbb{N}$ such that for any $n \geq n_0$ and each $\psi \in S_V$, 
\[
\lambda_{\varphi}^{-n}\|\mathcal{L}_{\varphi}^n(\psi)\|_\infty \leq r_0 < 1 \,.
\]

Hence, taking supremum on all the maps $\psi \in S_V$, by density of $\mathcal{H}_\alpha(\mathcal{B}(A, I))$ into $\mathcal{C}(\mathcal{B}(A, I))$, we obtain that for any $n \geq n_0$
\[
\|(\mathcal{L}_{\varphi}|_V)^n\|^{\frac{1}{n}} \leq r_0^{\frac{1}{n}} \lambda_{\varphi} < \lambda_{\varphi} \;.
\]

Therefore,  it follows that the spectral radius of the operator $\mathcal{L}_{\varphi}|_V$, denoted by $R(\mathcal{L}_{\varphi}|_V)$, is given by
\[
R(\mathcal{L}_{\varphi}|_V) := \inf_{n \in \mathbb{N}} \|(\mathcal{L}_{\varphi}|_V)^n\|^{\frac{1}{n}} \leq r_0^{\frac{1}{n_0}} \lambda_{\varphi} < \lambda_{\varphi}
\]

That is, for any eigenvalue $\lambda \neq \lambda_\varphi$ of $\mathcal{L}_{\varphi}$ we have the inequalities  $|\lambda| \leq r_0^{\frac{1}{n_0}} \lambda_{\varphi} < \lambda_{\varphi}$, which proves that the remainder of the spectrum is contained in a disk centered at zero with radius strictly smaller than $\lambda_\varphi$ and, thus, concludes the proof of item (4) of Theorem \ref{RPF-theorem}.

\victor{In order to finish the proof of Theorem \ref{RPF-theorem}, it is only necessary to prove existence of} $\lambda_{\varphi}$ maximal when the set of admissible sequences $\mathcal{B}(A, I)$ is topologically transitive. \victor{Remember that} $\mathcal{B}(A, I)$ admits a spectral decomposition. That is, there exist $k \in \mathbb{N}$ and a permutation ${\bf p}$ of the set $\{1, ..., k\}$ \victor{(which is actually a cycle of length $k$)}, such that 
\[
\mathcal{B}(A, I) = \mathcal{B}(A, I)_1 \cup ... \cup \mathcal{B}(A, I)_k \,,
\]
with $\mathcal{B}(A, I)_i$ closed for each $i \in \{1, ..., k\}$, $\mathcal{B}(A, I)_i \cap \mathcal{B}(A, I)_j = \emptyset$ when $i \neq j$, $\sigma(\mathcal{B}(A, I)_i) = \mathcal{B}(A, I)_{{\bf p}(i)}$ and each component $\mathcal{B}(A, I)_i$ is \victor{topologically mixing for the map $\sigma^k$.}

The maximality of $\lambda_{\varphi}$ is obtained from item (4) in the following way: let $\lambda_i$, with $i \in \{1, ..., k\}$, be the maximal isolated eigenvalue of the operator \victor{$\mathcal{L}^k_{\varphi}$ restricted to the component $\mathcal{B}(A, I)_i$ of the subshift $\mathcal{B}(A, I)$ (which is well defined because $\sigma^{k}(\mathcal{B}(A, I)_i) = \mathcal{B}(A, I)_i$). By item (4), $\lambda_i$ exists and there is a strictly positive function $f_i \in \mathcal{H}_\alpha(\mathcal{B}(A, I)_i)$ such that $\mathcal{L}^k_{\varphi}(f_i) = \lambda_i f_i$. Also using item (4) of this Theorem and the fact that for any pair $i, j \in \{1, ..., k\}$ there exists $n = n(i, j) < k$ such that $\sigma^{n}(\mathcal{B}(A, I)_i) = \mathcal{B}(A, I)_j$ (which is a consequence of the transitivity), we can prove that $\lambda_i = \lambda_0$ for any $i \in \{1, \hdots, k\}$. Define $\lambda_\varphi$ as the $k$-th positive root of $\lambda_0$. Since $\lambda_0$ is a simple eigenvalue of $\mathcal{L}^k_{\varphi}$ for the eigenfunction given by the equation $f_\varphi(x) := f_i(x)$ for each $x \in \mathcal{B}(A, I)_i$, with $i \in \{1, ..., k\}$, we have that $\lambda_\varphi$ results in a simple eigenvalue of $\mathcal{L}_{\varphi}$ for $f_\varphi$. Moreover, note that $f_\varphi$ is continuous because $\mathcal{B}(A, I)_i$ is closed for each $i \in \{1, ..., k\}$ and $\mathcal{B}(A, I)_i \cap \mathcal{B}(A, I)_j = \emptyset$ when $i \neq j$, which implies that $f_\varphi$ is strictly positive function belonging to $\mathcal{H}_\alpha(\mathcal{B}(A, I))$. 
%Then, the maximal eigenvalue associated to $\mathcal{L}_{\varphi}$ is $\lambda_{\varphi} := \max\{\widetilde{\lambda_i}:\; i \in \{1, ..., k\}\}$. 

We claim that the maximal  eigenvalue associated to $\mathcal{L}_{\varphi}$ is given by $\lambda_\varphi$. Indeed, if $\lambda$ is an eigenvalue of $\mathcal{L}_\varphi$ with $|\lambda| \geq \lambda_\varphi$, it follows that $\lambda_0 \geq |\lambda|^k \geq \lambda_\varphi^k = \lambda_0$, where the first one of the inequalities follows from the maximality of $\lambda_0$ (since $\lambda^k$ is an eigenvalue for $\mathcal{L}^k_{\varphi}$), and the last equality follows from the definition of $\lambda_{\varphi}$. 
The above implies that $|\lambda| = \lambda_\varphi$ and} concludes the proof of item (1) and, thus, the proof of Theorem \ref{RPF-theorem}.
\end{proof}

One of the main utilities of Theorem \ref{RPF-theorem} is that offers a new approach to prove a Ruelle's Perron-Frobenius Theorem in the setting of countable Markov shifts under similar hypotheses on the dynamics of the subshift that the ones assumed in \victor{\cite{MR1853808}, \cite{MR2003772} and \cite{MR1738951}.} We present below the proof of the mentioned result. 

\begin{proof}[Proof of Proposition \ref{RPF-theorem-SFT}]
%\victor{Os primeiros dois parágrafos aparecem com algumas modificaçoes embaixo, porque ficava meio estranho no começo, aparece depois as propriedades do mapa i}
Consider the compact set $M = \{b_k : k \in \mathbb{N}\} \cup \{b_{\infty}\}$ and let $A : M \times M \to \mathbb{R}$ be a continuous function satisfying the following for each $i, j \in \mathbb{N}$: 
\begin{enumerate}[i)]
\item $A(b_i, b_j) = 1$ if and only if ${\bf A}_{b_i, b_j} = 1$; 
\item $A(b_i, b_j) \neq 1$ if and only if ${\bf A}_{b_i, b_j} = 0$. 
\end{enumerate}

Note that $A$ satisfies the equation ${\bf A}_{b_i, b_j} = ({\bf 1}_{\{1\}} \circ A)(b_i, b_j)$ and its existence is guaranteed by Urysohn's Lemma and the property that there is $j_0 \in \mathbb{N}$ such that for any $j \geq j_0$, we have ${\bf A}_{b_i, b_j} = {\bf A}_{b_i, b_{j_0}}$ for all $i \in \mathbb{N}$.

By continuity of $A$, it is guaranteed that 
$$
A(b_i, b_{\infty}) = \lim_{j \to \infty}A(b_i, b_j) = A(b_i, b_{j_0})$$ and 
$$A(b_{\infty}, b_j) = \lim_{i \to \infty}A(b_i, b_j)\;.$$

Since there is $j_0 \in \mathbb{N}$ such that for any $j \geq j_0$, ${\bf A}_{b_i, b_j} = {\bf A}_{b_i, b_{j_0}}$ holds for all $i \in \mathbb{N}$, it follows that the map $s$ assigning to each $a \in M$ its corresponding section $s(a)$ in $A^{-1}(\{1\})$ is a locally constant map.

%\victor{Nesta parte troque \varphi por \psi, porque \varphi é fixada no enunciado}
Therefore, we can extend any $\psi \in \mathcal{H}_{\alpha}(\Sigma_{\bf A})$ to a function $\psi' : \mathcal{B}(A, \{1\}) \to \mathbb{R}$, defined as the map that assigns to each point $x \in \mathcal{B}(A, \{1\})$ the value 
\begin{equation}
\label{limit-extension}
\psi'(x) = \lim_{\substack{y \to x \\ y \in \Sigma_{\bf A}}}\psi(y) \,.
\end{equation}
 
Such limit exists because the function $\psi$ belongs to the set $\mathcal{H}_{\alpha}(\Sigma_{\bf A})$. Moreover,  $\psi' \in \mathcal{H}_{\alpha}(\mathcal{B}(A, \{1\}))$ because, for any pair $x, y \in \mathcal{B}(A, \{1\})$ we can choose sequences $(x^n)_{n \in \mathbb{N}}$ and $(y^n)_{n \in \mathbb{N}}$ taking values in $\Sigma_{\bf A}$ such that $\lim_{n \to \infty}x^n = x$ and $\lim_{n \to \infty} y^n = y$, thus, by \eqref{limit-extension} and continuity of $\psi$, it follows that \victor{
\begin{equation}
\label{Holder-extension}
|\psi'(x) - \psi'(y)| 
= \lim_{n \to \infty}\lim_{m \to \infty}|\psi(x^n) - \psi(y^n)| 
%\leq \lim_{n \to \infty}\lim_{m \to \infty}\mathrm{Hol}_{\psi}d(x^n, y^m) 
\leq \mathrm{Hol}_{\psi}d(x, y) \,,
\end{equation} }  
which implies our assertion and that $\mathrm{Hol}_{\psi'} = \mathrm{Hol}_{\psi}$. 

Define the operator $i : \mathcal{H}_{\alpha}(\Sigma_{\bf A}) \to \mathcal{H}_{\alpha}(\mathcal{B}(A, \{1\}))$ by the equation 
\[
i(\psi) = \psi' \;,
\] 
with $\psi'$ of the form in \eqref{limit-extension}. 

We claim that the operator $i$ is an isometric isomorphism with inverse given by $i^{-1}(\psi') = \psi'|_{\Sigma_{\bf A}}$ for any $\psi' \in \mathcal{H}_{\alpha}(\mathcal{B}(A, \{1\}))$ and  the equation $i \circ L_{\phi} = \mathcal{L}_{\varphi'} \circ i$ is satisfied, when $\phi$, $\varphi$ and $p$ are such as appears in Proposition \ref{RPF-theorem-SFT} and $\nu = \sum_{k = 1}^{\infty}p_k \delta_{b_k}$ is the a priori probability measure associated to $\mathcal{L}_{\varphi'}$.
%\victor{A parte que apague segue da definição do mapa i, porém agora aparece a definição do mapa \phi' embaixo} 

Indeed, by \eqref{Holder-extension}, it follows that $\mathrm{Hol}_{i(\psi)} = \mathrm{Hol}_{\psi}$ and, by \eqref{limit-extension}, we have $\|i(\psi)\|_{\infty} = \|\psi\|_{\infty}$, thus, $\|i(\psi)\|_{\alpha} = \|\psi\|_{\alpha}$, i.e., the operator $i$ is an isometry. It is not difficult to check that $i$ is injective and for any $\widetilde{\psi} \in \mathcal{H}_{\alpha}(\mathcal{B}(A, \{1\}))$ we have $i(\widetilde{\psi}|_{\Sigma_{\bf A}}) = \widetilde{\psi}$, which implies that the map is sobrejective and that $i^{-1}(\widetilde{\psi}) = \widetilde{\psi}|_{\Sigma_{\bf A}}$, moreover, the foregoing implies that $\widetilde{\psi}$ is of the form in \eqref{limit-extension}, that is, $\widetilde{\psi} = \psi'$ for some $\psi \in \mathcal{H}_{\alpha}(\Sigma_{\bf A})$. The continuity of the maps $i$ and $i^{-1}$ is a direct consequence of linearity and the equality $\|i(\varphi)\|_{\alpha} = \|\varphi\|_{\alpha}$, then, $i$ is an isomorphism. 

%\victor{O paragrafo que aparece embaixo é similar do que aparecia nos primeiros dois parágrafos da demonstração, mais nesta parte fica numa posição mais adequada}
Since $\phi = \varphi + \log(p \circ \pi_1)$ for some $\varphi \in \mathcal{H}_{\alpha}(\Sigma_{\bf A})$ and $p$ such that $p(b_k) > 0$ for each $k \in \mathbb{N}$ and $\sum_{k = 1}^{\infty}p(b_k) = 1$, we can define a function $\phi' : \mathcal{B}(A, \{1\}) \to \mathbb{R}$ given by
\[
\phi' := \varphi' + \log(p' \circ \pi_1) \,,
\]
 where $\varphi' = i(\varphi)$ and $p'$ is a function from $M$ into $[0, 1]$ defined as $p'(b_k) = p(b_k)$ for each $k \in \mathbb{N}$ and $p'(b_{\infty}) = \lim_{k \to \infty} p(b_k) = 0$. It is not difficult to check that $\phi' \in \mathcal{H}_{\alpha}(\mathcal{B}(A, \{1\}))$. Moreover, this function provides a connection between the operators $i \circ L_{\phi}$ and $\mathcal{L}_{\varphi'} \circ i$, as we will show below:

Indeed, if $\psi \in \mathcal{H}_{\alpha}(\Sigma_{\bf A})$ and $x \in \mathcal{B}(A, \{1\})$, with $x_1 \neq b_{\infty}$, we have \victor{
\begin{equation*}
(\mathcal{L}_{\varphi'} \circ i)(\psi)(x) 
= \mathcal{L}_{\varphi'}(\psi')(x) \nonumber \\
%= \sum_{a \in s(x_1)} e^{\varphi'(ax)}\psi'(ax)p'(a) 
= \sum_{a \in s(x_1) \setminus \{b_{\infty}\}} e^{\phi'(ax)}\psi'(ax) \,.
\end{equation*} }

Besides that, from the fact that for any of point $y \in \Sigma_{\bf A}$ close enough to $x$ we have $s(x_1) = s(y_1)$, it follows that \victor{
\begin{align}
(i \circ L_{\phi})(\psi)(x) 
&= i(L_{\phi}(\psi))(x) \nonumber \\
&= \lim_{\substack{y \to x \\ y \in \Sigma_{\bf A}}}\Bigl(\sum_{\substack{a \in M_0 \\ {\bf A}_{a, x_1} = 1}} e^{\varphi(ay) + \log(p(a))}\psi(ay) \Bigr) \nonumber \\
%&= \lim_{\substack{y \to x \\ y \in \Sigma_{\bf A}}}\Bigl(\sum_{a \in s(x_1) \setminus \{b_{\infty}\}} e^{\varphi(ay) + \log(p(a))}\psi(ay) \Bigr) \nonumber \\
&= \sum_{a \in s(x_1) \setminus \{b_{\infty}\}} e^{\varphi'(ax) + \log(p'(a))}\psi'(ax)  
= \sum_{a \in s(x_1) \setminus \{b_{\infty}\}} e^{\phi'(ax)}\psi'(ax) \nonumber \,.
\end{align} }

Therefore, by continuity of the functions $(\mathcal{L}_{\varphi'} \circ i)(\psi)$ and $(i \circ L_{\phi})(\psi)$, it follows that
\[
(\mathcal{L}_{\varphi'} \circ i)(\psi)(x) = (i \circ L_{\phi})(\psi)(x) \,
\]
for each $x \in \mathcal{B}(A, \{1\})$, thus, we have that $i \circ L_{\phi} = \mathcal{L}_{\varphi'} \circ i$, such as we wanted to prove.

It is widely known that any countable Markov shift $\Sigma_{\bf A}$ with irreducible matrix {\bf A} admits a decomposition of the form
\[
\Sigma_{\bf A} = \Sigma_{\bf A}^1 \cup ... \cup \Sigma_{\bf A}^p \;.
\]

Where $p$ is the period of the matrix ${\bf A}$, all the sets $\Sigma_{\bf A}^k$, with $k \in \{1, ..., p\}$, are pairwise disjoint and any component $\Sigma_{\bf A}^k$ is a countable Markov shift with aperiodic incidence matrix ${\bf A}^{(k)}$ (see for instance Remark 7.1.35 in \cite{MR1484730}). Moreover, in this case also is satisfied that ${\bf A}^{(k)}_{b_i, b_j} = ({\bf A}^p)_{b_i, b_j}$ for each $b_j \in \mathcal{C}_k$, where
\[
\mathcal{C}_k := \{b_j :\; ({\bf A}^{np + k})_{b_i, b_j} > 0 \text{ for some } n\in \mathbb{N} \} \;.
\]

The above induces a spectral decomposition in the subshift $\mathcal{B}(A, \{1\})$ of the form
\[
\mathcal{B}(A, \{1\}) = \mathcal{B}(A, \{1\})_1 \cup ... \cup \mathcal{B}(A, \{1\})_p \;,
\]
that satisfies the hypothesis of Theorem \ref{RPF-theorem}.

On other hand, by item (1) of Theorem \ref{RPF-theorem}, since $\varphi' \in \mathcal{H}_{\alpha}(\mathcal{B}(A, \{1\}))$, it is guaranteed existence of a $\lambda_{\phi'} > 0$ and a strictly positive function $f_{\phi'} \in \mathcal{H}_{\alpha}(\mathcal{B}(A, \{1\}))$ such that 
\begin{equation}
\label{eigenfunction-SFT}
\mathcal{L}_{\varphi'}(f_{\phi'}) = \lambda_{\phi'}f_{\phi'} \,.
\end{equation}  

Note that in fact $\lambda_{\phi'} = \lambda_{\varphi'}$ and $f_{\phi'} = f_{\varphi'}$ in the notation of Theorem \ref{RPF-theorem} i.e. $\varphi' = \phi' + \log(p' \circ \pi_1)$ . However, in this proof it is more convenient to use the notation that appears in \eqref{eigenfunction-SFT}.

Taking $\lambda_{\phi} = \lambda_{\phi'}$ and $f_{\phi} = f_{\phi'}|_{\Sigma_{\bf A}}$, we assert that $f_{\phi}$ is an eigenfunction of $L_{\phi}$ associated to the eigenvalue $\lambda_{\phi}$, that is, $L_{\phi}f_{\phi} = \lambda_{\phi}f_{\phi}$. 

Indeed, by the above definition we have $i(f_{\phi}) = f_{\phi'}$, which implies that 
\[
L_{\phi}f_{\phi} = i^{-1}(\mathcal{L}_{\varphi'}(f_{\phi'})) = i^{-1}(\lambda_{\phi}f_{\phi'}) = \lambda_{\phi}i^{-1}(f_{\phi'}) = \lambda_{\phi}f_{\phi} \,.
\]

Moreover, following the argument of the proof of Theorem \ref{RPF-theorem}, we can prove that $\lambda_{\phi}$ is simple and maximal, which concludes the proof of item 1. of Proposition \ref{RPF-theorem-SFT}.

On other hand, by item (2) of Theorem \ref{RPF-theorem}, there is a Borelian probability measure $\rho_{\phi'}$ defined on the Borelian sets of $\mathcal{B}(A, \{1\})$ satisfying the equation $\mathcal{L}^*_{\varphi'}(\rho_{\phi'}) = \lambda_{\phi}\rho_{\phi'}$. In fact $\rho_{\phi'} = \rho_{\varphi'}$ in the notation of Theorem \ref{RPF-theorem}, nevertheless, by simplicity we will use the notation proposed in this proof.

Define $\rho_{\phi'} \circ i : \mathcal{H}_{\alpha}(\Sigma_{\bf A}) \to \mathbb{R}$ as the linear functional assigning to each $\psi \in \mathcal{H}_{\alpha}(\Sigma_{\bf A})$ the value 
\[
(\rho_{\phi'} \circ i)(\psi) = \rho_{\phi'}(i(\psi)) = \rho_{\phi'}(\psi') \;.
\]

Note that as a consequence of the characterization of the weak* topology by H\"older continuous functions and the fact that $(\rho_{\phi'} \circ i)(1) = 1$, it follows that $\rho_{\phi'} \circ i$ define a Borelian probability measure on $\Sigma_{\bf A}$. Hereafter, we will use the following notation for such probability measure $\rho_{\phi} := \rho_{\phi'} \circ i$.

We claim that $\rho_{\phi}$ is an eigenprobability of the operator $L^*_{\phi}$ associated to the eigenvalue $\lambda_{\phi}$, that is, $L^*_{\phi}(\rho_{\phi}) = \lambda_{\phi}\rho_{\phi}$. 

Indeed, for any $\psi \in \mathcal{H}_{\alpha}(\Sigma_{\bf A})$ we have\victor{
\begin{align}
L^*_{\phi}(\rho_{\phi})(\psi) = \rho_{\phi}(L_{\phi}(\psi)) 
%&= \rho_{\phi'}((i \circ L_{\phi})(\psi)) \nonumber \\
%&= \rho_{\phi'}((\mathcal{L}_{\varphi'} \circ i)(\psi)) \nonumber \\
&= \rho_{\phi'}(\mathcal{L}_{\varphi'}(\psi')) \nonumber \\
&= \mathcal{L}^*_{\varphi'}(\rho_{\phi'})(\psi') = \lambda_{\phi}\rho_{\phi'}(\psi') = \lambda_{\phi}\rho_{\phi}(\psi ) \nonumber \,.
\end{align} }

The above concludes the proof of the item (2) of Proposition \ref{RPF-theorem-SFT}.

Since $L_{\overline{\phi}}(\psi) = \frac{1}{\lambda_{\phi}}L_{\phi}(\psi f_{\phi})$ \victor{for any $\psi \in \mathcal{H}_{\alpha}(\Sigma_{\bf A})$, 
%following a similar procedure to the one that appears in \eqref{fixed-point} and \eqref{invariant-Gibbs-state}, 
it is not difficult to show that the measure $d\mu_{\phi} = f_{\phi}d\rho_{\phi}$,} which we can assume w.l.o.g. a probability measure, is a fixed point for the operator $L^*_{\phi}$ and is $\sigma$-invariant. The foregoing concludes the proof of item (3) of Proposition \ref{RPF-theorem-SFT}. 

In order to prove the item (4) of Proposition \ref{RPF-theorem-SFT}, first note that the aperiodicity of the matrix ${\bf A}$ implies that the set $\mathcal{B}(A, \{1\})$ is topologically mixing, which implies item (4) of Theorem \ref{RPF-theorem} for the extension of any bounded potential in $\mathcal{H}(\Sigma_{\bf A})$.

On other hand, since $i : \mathcal{H}_{\alpha}(\Sigma_{\bf A}) \to \mathcal{H}_{\alpha}(\mathcal{B}(A, \{1\}))$ is an isomorphism, it follows that for any $\psi \in \mathcal{H}_{\alpha}(\Sigma_{\bf A})$, we have $i^{-1}(\psi') = \psi'|_{\Sigma_{\bf A}} = \psi$ and $\rho_{\phi'}(\psi') = \rho_{\phi}(i^{-1}(\psi')) = \rho_{\phi}(\psi'|_{\Sigma_{\bf A}}) = \rho_{\phi}(\psi)$. The above, joint with the fact that $f_{\phi} = i^{-1}(f_{\phi'}) = f_{\phi'}|_{\Sigma_{\bf A}}$, implies that \victor{
\begin{equation}
\label{prime}
\mu_{\phi'}(\psi') = \rho_{\phi'}(\psi'f_{\phi'}) 
%= \rho_{\phi}(i^{-1}(\psi'f_{\phi'})) = \rho_{\phi}((\psi'f_{\phi'})|_{\Sigma_{\bf A}}) 
= \rho_{\phi}(\psi f_{\phi})) = \mu_{\phi}(\psi) \,.
\end{equation} }

Thus, by item (4) of Theorem \ref{RPF-theorem}, it follows that for each $x \in \Sigma_{\bf A}$, \victor{
\begin{align}
\mu_{\phi}(\psi)
&= \mu_{\phi'}(\psi') \nonumber \\
&= \lim_{n \to \infty} \mathcal{L}^n_{\overline{\varphi'}}(\psi')(x) \nonumber \\
&= \lim_{n \to \infty} \sum_{a_n \in s(a_{n-1})} \ldots \sum_{a_1 \in s(x_1)}e^{S_n \overline{\varphi'}(a^n x)}\psi'(a^n x) p'(a_1) \ldots p'(a_n) \nonumber \\ 
%&= \lim_{n \to \infty} \sum_{a_n \in s(a_{n-1}) \setminus \{b_{\infty}\}} \ldots \sum_{a_1 \in s(x_1) \setminus \{b_{\infty}\}}e^{S_n \overline{\varphi'}(a^n x)}\psi'(a^n x) p'(a_1) \ldots p'(a_n) \nonumber \\
&= \lim_{n \to \infty} \sum_{\substack{a_n \in M_0 \\ {\bf A}_{a_n, a_{n-1}} = 1}} \ldots \sum_{\substack{a_1 \in M_0 \\ {\bf A}_{a_1, x_1} = 1}}e^{S_n \overline{\varphi}(a^n x)}\psi(a^n x) p(a_1) \ldots p(a_n) \nonumber \\
&= \lim_{n \to \infty} L^n_{\overline{\phi}}(\psi)(x) \nonumber \,.
\end{align} }

Moreover, since the limit $\mu_{\phi'}(\psi') = \lim_{n \to \infty} \mathcal{L}^n_{\overline{\varphi'}}(\psi')$ is uniform on $\mathcal{B}(A, \{1\})$ in the norm $\|\cdot\|_\infty$ by item (4) of Theorem \ref{RPF-theorem}, it follows that $\mu_{\phi}(\psi) = \lim_{n \to \infty} L^n_{\overline{\phi}}(\psi)$ uniformly on $\Sigma_{\bf A}$ in the norm $\|\cdot\|_\infty$ as well, which concludes the proof of item (4) of Proposition \ref{RPF-theorem-SFT}.

The proof of maximality of $\lambda_{\phi}$ when the matrix ${\bf A}$ is irreducible and the proof of the spectral gap when ${\bf A}$ is aperiodic follow in a similar way that in the proof of Theorem \ref{RPF-theorem}. The above concludes the proof of Proposition \ref{RPF-theorem-SFT}.

\end{proof}

In \cite{MR1738951} appears a characterization of a class of potentials in which the Ruelle's Perron-Frobenius Theorem holds in the setting of topologically mixing countable Markov shifts. This class of potentials are the so called positive recurrent potentials. Below we will show that the potential $\phi$  defined in Proposition \ref{RPF-theorem-SFT} belongs to such class. 

Consider $a \in M_0$ and $[a] = \{x \in \Sigma_{\bf A} : x_1 = a\}$. Also define 
\[
Z_n(\phi, a) := \sum_{\substack{\sigma^n(y) = y \\ y \in \Sigma_{\bf A}}}e^{S_n \phi(y)}{\bf 1}_{[a]}(y) \,.
\]

Under the assumption that $\Sigma_{\bf A}$ is topologically mixing, we say that the potential $\phi$ is {\bf positive recurrent}, if there are $N_a \in \mathbb{N}$ and $C_a > 0$, such that
\[
\frac{1}{\lambda^n_{\phi}}Z_n(\phi, a) \in [C_a^{-1}, C_a] \,,
\]
 for each $n \geq N_a$.   
 
As a consequence of the definition of $\overline{\phi}$ and the item (4) in Proposition \ref{RPF-theorem-SFT}, it follows that for any $x \in [a]$ we have
\[
0 < \inf\{f_{\phi}(y) : y \in \Sigma_{\bf A}\} \leq \lim_{n \to \infty} \frac{1}{\lambda^n_{\phi}}L^n_{\phi}(1)(x) = \lim_{n \to \infty} \frac{1}{\lambda^n_{\phi}} \sum_{\substack{a^n \in M_0^n \\ a^nx \in \Sigma_{\bf A}}}e^{S_n \phi(a^nx)} \,. 
\]

By the above, there are $N_1 \in \mathbb{N}$ and $C_1 > 0$, such that
\begin{equation}
\label{bound-sum-1}
\frac{1}{\lambda^n_{\phi}} \sum_{\substack{a^n \in M_0^n \\ a^nx \in \Sigma_{\bf A}}}e^{S_n \phi(a^nx)} \in [C_1^{-1}, C_1] \,,
\end{equation}
for each $n \geq N_1$.

On other hand, since $\Sigma_{\bf A}$ is topologically mixing, there exists $N_2 \in \mathbb{N}$ such that $\sigma^{-n}([a]) \cap [a] \neq \emptyset$ for each $n \geq N_2$, thus, there is a periodic point of the form $\widetilde{y}^n = \overline{a b_n \ldots b_1}$. Moreover, since ${\bf A}_{b_1, a} = 1$, it follows that $b^nx \in \Sigma_{\bf A}$, with $b^n = b_n, \ldots, b_1$.

Choosing $N_a = \max\{N_1, N_2\}$, we have that for each $n \geq N_a$, equation \eqref{bound-sum-1}, and the following inequality
\[
|S_n\phi(b^nx) - S_n\phi(\sigma(\widetilde{y}^n))| \leq \frac{2^{\alpha}}{2^{\alpha} - 1}\mathrm{Hol}_{\varphi} = C_2 \,.
\]
are satisfied.

Besides that, for any $a^n \in M_0^n$, such that, $a^nx \in \Sigma_{\bf A}$,  we have
\[
|S_n\phi(b^nx) - S_n\phi(a^nx)| \leq C_2 \,,
\]
thus,
\[
|S_n\phi(a^nx) - S_n\phi(\sigma(\widetilde{y}^n))| \leq 2C_2 \,.
\]

The foregoing implies that
\begin{align}
|S_n\phi(a^nx) - S_{n+1}\phi(\widetilde{y}^n)| 
&= |S_n\phi(a^nx) - S_{n+1}\phi(\sigma(\widetilde{y}^n))| \nonumber \\
&\leq 2C_2 + \|\varphi\|_{\infty} + \log(p(a)) = C_3 \nonumber \,.
\end{align}

Then, the following inequalities are satisfied 
\[
\sum_{\substack{a^n \in M_0^n \\ a^nx \in \Sigma_{\bf A}}}e^{S_n \phi(a^nx) - C_3} \leq Z_{n+1}(\phi, a) \leq \sum_{\substack{a^n \in M_0^n \\ a^nx \in \Sigma_{\bf A}}}e^{S_n \phi(a^nx) + C_3} \,,
\]
which, by \eqref{bound-sum-1}, is equivalent to say that
\begin{equation}
\frac{1}{\lambda^{n+1}_{\phi}}Z_{n+1}(\phi, a) \in [(C_1e^{C_3})^{-1}, C_1e^{C_3}] \,.
\end{equation}

Therefore, taking $C_a = C_1e^{C_3}$ we obtain that $\phi$ is positive recurrent, such as we wanted to prove.

Now we will prove a variational principle of the pressure, with the aim to show that the Gibbs states found in Theorem \ref{RPF-theorem} and in Proposition \ref{RPF-theorem-SFT} result in equilibrium states, that is, $\sigma$-invariant probability measures that optimize the energy of the system, which since a theoretical approach are the observables that attain the supremum in the variational principle. In order to do that, we will introduce a definition of entropy, which has been widely studied in another settings (see for instance \cite{MR3377291, 1LoVa19}).

Given a $\sigma$-invariant probability measure $\mu$, we define the {\bf entropy of $\mu$} as
\[
h(\mu) := \inf \bigl\{\mu(\log(\mathcal{L}_0(u)) - \log(u)) : u \in \mathcal{C}^+(\mathcal{B}(A, I)) \bigr\} \,.
\] 

\victor{It is easy to check that the entropy map is upper semi-continuous and $h(\mu) \leq 0$ for any $\mu \in \mathcal{M}_\sigma(\mathcal{B}(A, I))$. Moreover, $h(\mu) < 0$ for any $\mu \in \mathcal{M}_\sigma(\mathcal{B}(A, I))$ when $\mathcal{B}(A, I) \subsetneq M^{\mathbb{N}}$ and} the entropy of the Gibbs state $\mu_\varphi$ associated to a potential $\varphi \in \mathcal{H}_{\alpha}(\mathcal{B}(A, I))$ satisfies the following equation  
\[
h(\mu_{\varphi}) = -\mu_{\varphi}(\overline{\varphi}) \,, 
\] 
which guarantees that the supremum that appears in the variational principle below is in fact attained in the Gibbs state $\mu_{\varphi}$ \victor{(see for instance \cite{1LoVa19}).}

\begin{lemma}
\label{variational-principle}
\victor{Consider} a potential $\varphi \in \mathcal{H}_{\alpha}(\mathcal{B}(A, I))$ and $\mu_{\varphi}$ a Gibbs state associated to $\varphi$. Then, the following variational principle is satisfied: 
\[
\log(\lambda_{\varphi}) = h(\mu_{\varphi}) + \mu_{\varphi}(\varphi) = \sup \bigl\{ h(\mu) + \mu(\varphi) : \mu \in \mathcal{M}_{\sigma}(\mathcal{B}(A, I)) \bigr\} \,.
\] 
\end{lemma}
\begin{proof}
The proof of this Lemma can be obtained following a similar procedure as in the proofs of Lemma 2.7 and Theorem 2.8 in \cite{1LoVa19}.
\end{proof}

As a consequence of the lemma above, we have that \victor{
\begin{equation}
\label{entropybounded}
 \sup \bigl\{ h(\mu) : \mu \in \mathcal{M}_{\sigma}(\mathcal{B}(A, I)) \bigr\} \leq 0,
\end{equation} }
which will be a necessary result in the proof of Proposition \ref{ground-states} below.
 
By compactness of the set $\mathcal{B}(A, I)$, it is guaranteed existence of $\varphi$-maximizing measures associated to a potential $\varphi \in \mathcal{H}_{\alpha}(\mathcal{B}(A, I))$ as accumulation points in the weak* topology of the family of equilibrium states $(\mu_{t\varphi})_{t>1}$, which are known as ground-states (see for instance \cite{MR3114331}, \cite{MR2818689} and \cite{2LoVa19}). The above, is the main tool that we will use in the proof of Proposition \ref{ground-states}.

\begin{proof}[Proof of Proposition \ref{ground-states}]
Consider the family of equilibrium states $(\mu_{\phi_t})_{t>1}$. Note that existence of the equilibrium state $\mu_{\phi_t}$ for each $t > 1$ is a consequence of aperiodicity of the matrix ${\bf A}$. Since $(t\varphi)' = t\varphi'$ for each $t > 1$, by \eqref{prime}, for any $\psi \in \mathcal{H}_{\alpha}(\Sigma_{\bf A})$ and each $t > 1$  we have $\mu_{{\phi_t}}(\psi) =  \mu_{{\phi'_t}}(\psi')$, where $\phi'_t := t\varphi' + \log(p' \circ \pi_1)$. Besides that, as $\mathcal{B}(A, \{1\})$ is a compact set, there is an accumulation point $\mu'_{\infty}$ of the family $(\mu_{\phi'_t})_{t>1}$ when $t \to \infty$. That is, there exists a sequence $(t_n)_{n \in \mathbb{N}}$ with $\lim_{n \to \infty}t_n = \infty$, such that each $\psi' \in \mathcal{H}_{\alpha}(\mathcal{B}(A, \{1\}))$ satisfies 
\begin{equation}
\label{limit-equilibrium-states}
\mu'_{\infty}(\psi') = \lim_{n \to \infty}\mu_{\phi'_{t_n}}(\psi') = \lim_{n \to \infty}\mu_{\phi_{t_n}}(\psi) \,.
\end{equation}

Define $\mu_{\infty} := \mu'_{\infty} \circ i$. Then, as a consequence of the characterization of the weak* topology by H\"older continuous functions, it follows that $\mu_{\infty}$ is a Borelian probability measure on $\Sigma_{\bf A}$. Moreover, by \eqref{limit-equilibrium-states}, it follows that for any $\psi \in \mathcal{H}_{\alpha}(\Sigma_{\bf A})$ we have 
\[
\mu_{\infty}(\psi) = \mu'_{\infty}(i(\psi)) = \mu'_{\infty}(\psi') = \lim_{n \to \infty}\mu_{\phi_{t_n}}(\psi) \,,
\]
which implies that $\lim_{n \to \infty}\mu_{\phi_{t_n}} = \mu_{\infty}$ in the weak* topology. On other hand, by Lemma \ref{variational-principle}, for each $t > 1$, 
\[
h(\mu_{\phi'_t}) + t \mu_{\phi'_t}(\varphi') = \sup \bigl\{ h(\mu) + t \mu( \varphi') : \mu \in \mathcal{M}_{\sigma}(\mathcal{B}(A, \{1\})) \bigr\} \,,
\]
which implies, using \eqref{entropybounded}, that $\mu'_{\infty} = \lim_{n \to \infty}\mu_{\phi'_{t_n}}$ is a $\varphi'$-maximizing measure (see for instance \cite{MR3114331} and \cite{2LoVa19}). Thus, for any $\mu' \in \mathcal{M}_{\sigma}(\mathcal{B}(A, \{1\}))$, it follows
\[
\mu_{\infty}(\varphi) = \mu'_{\infty}(i(\varphi)) = \mu'_{\infty}(\varphi') \geq \mu'(\varphi') \,.
\]

In particular, for any $\mu \in \mathcal{M}_{\sigma}(\mathcal{B}(A, \{1\}))$ such that $\mathrm{supp}(\mu) \subset \Sigma_{\bf A}$, we have 
\[
\mu_{\infty}(\varphi) \geq \mu(\varphi') = \mu(\varphi'|_{\Sigma_{\bf A}}) = \mu(\varphi) \,
\]
and that means $\mu_{\infty}$ is a $\varphi$-maximizing measure, as we wanted to prove.
\end{proof}

\begin{remark}
Since any countable Markov shift $\Sigma_{\bf A}$ with irreducible incidence matrix ${\bf A}$ admits a spectral decomposition (see for details the proof of Proposition \ref{RPF-theorem-SFT}), it follows that the results obtained in Proposition \ref{ground-states} can be extended to the case of countable Markov shifts with irreducible incidence matrix.
\end{remark}

\section{Involution Kernel}
\label{involution-kernel-section}

The involution kernel is a useful tool to find maximizing measures in bilateral topological subshifts from the theory of transfer operators, because, joint with the Livsic's Theorem, provides a connection between bilateral and unilateral topological subshifts via cohomology. \rafa {In this section we present the proof of Theorem \ref{bilateral-extension}, where the involution kernel is used to characterize the normalized eigenfunction of the Ruelle operator associated to the maximal eigenvalue, in terms of the eigenprobabilities $\rho_{\varphi}$ and $\rho_{\varphi^*}$, 
	given by Theorem \ref{RPF-theorem}.}
	%, through a characterization involving the integral of a function that depends on the kernel with respect to the eigenprobability of the corresponding dual of the Ruelle operator. 

In order to prove Theorem \ref{bilateral-extension} it is necessary to prove the following Lemma:

\begin{lemma}
\label{transpose-Ruelle-operator}
Let $\varphi \in \mathcal{H}_{\alpha}(\mathcal{B}(A, I))$, $W \in \mathcal{H}_{\alpha}(\widehat{\mathcal{B}(A, I)})$ and $\varphi^* \in \mathcal{H}_{\alpha}(\mathcal{B}(A, I)^*)$ be potentials satisfying \eqref{dual-potential}. Then, for any pair $(y, x) \in \mathcal{B}(A, I)^* \times \mathcal{B}(A, I)$, we have 
\[
\mathcal{L}_{\varphi^*}(({\bf 1}_I \circ A \circ \pi_{1, 1})(\cdot, x)e^{W(\cdot | x)})(y) = \mathcal{L}_{\varphi}(({\bf 1}_I \circ A \circ \pi_{1, 1})(y, \cdot)e^{W(y | \cdot)})(x) \,.
\]
\end{lemma}
\begin{proof}
Note that any function $\psi \in \mathcal{H}_{\alpha}(\widehat{\mathcal{B}(A, I)})$ can be extended to a bounded function from $\mathcal{B}(A, I)^* \times \mathcal{B}(A, I)$ into $\mathbb{R}$, which we will denote by $\psi$ as well. Moreover, in this case, for any $V \in \mathcal{H}_{\alpha}(\widehat{\mathcal{B}(A, I)})$, the function $e^{V}\psi$ can be extended to a bounded function defined on the set $\mathcal{B}(A, I)^* \times \mathcal{B}(A, I)$, in such a way that the function $({\bf 1}_I \circ A \circ \pi_{1, 1})e^{V}\psi$ is bounded on the set $\mathcal{B}(A, I)^* \times \mathcal{B}(A, I)$ and it is equal to $0$ for each point in the set $(\mathcal{B}(A, I)^* \times \mathcal{B}(A, I)) \setminus \widehat{\mathcal{B}(A, I)}$.
 
In particular, taking $\psi \equiv 1$, $V = W$, with $W$ satisfying \eqref{dual-potential}, and using that $\widehat{\varphi}(y|x) = \varphi(x)$ and $\widehat{\varphi}^*(y|x) = \varphi^*(y)$ for each $(y|x) \in \widehat{\mathcal{B}(A, I)}$, it follows that for each $(y, x) \in \mathcal{B}(A, I)^* \times \mathcal{B}(A, I)$  \victor{
\begin{align}
&\mathcal{L}_{\varphi^*}(({\bf 1}_I \circ A \circ \pi_{1, 1})(\cdot, x)e^{W(\cdot | x)})(y) \nonumber \\
%&= \int_{s^*(y_1)}e^{\varphi^*(ya) + W(ya|x)}({\bf 1}_I \circ A)(a, x_1) d\nu(a) \nonumber \\
&= \int_M e^{\widehat{\varphi}^*(ya|x) + W(ya|x)}({\bf 1}_I \circ A)(y_1, a)({\bf 1}_I \circ A)(a, x_1) d\nu(a) \nonumber \\
&= \int_M e^{\widehat{\varphi}(y|ax) + W(y|ax)}({\bf 1}_I \circ A)(y_1, a)({\bf 1}_I \circ A)(a, x_1) d\nu(a) \nonumber \\
%&= \int_{s(x_1)} e^{\varphi(ax) + W(y|ax)}({\bf 1}_I \circ A)(y_1, a) d\nu(a) \nonumber \\
&= \mathcal{L}_{\varphi}(({\bf 1}_I \circ A \circ \pi_{1, 1})(y, \cdot)e^{W(y | \cdot)})(x) \nonumber \,.
\end{align}}
where in the third equality we use \eqref{involution-kernel}, which is equivalent to dual-potential.
\end{proof}

The proof of Theorem \ref{bilateral-extension} is such as follows:

\begin{proof}[Proof of Theorem \ref{bilateral-extension}]
Consider $\rho_{\varphi^*}$ the eigenprobability associated to $\mathcal{L}^*_{\varphi^*}$, which is defined on the Borelian sets in $\mathcal{B}(A, I)^*$. Define $f$ as the map assigning to each $x \in \mathcal{B}(A, I)$ the value $f(x) = \rho_{\varphi^*}(({\bf 1}_I \circ A \circ \pi_{1, 1})(y, x)e^{W(y| x) - c})$. Then, we have the following: 
\begin{align}
f(x) 
&= \rho_{\varphi^*}(({\bf 1}_I \circ A \circ \pi_{1, 1})(y, x)e^{W(y|x) - c}) \nonumber \\
&= \frac{1}{\lambda_{\varphi^*}}\mathcal{L}^*_{\varphi^*}(\rho_{\varphi^*})(({\bf 1}_I \circ A \circ \pi_{1, 1})(y, x)e^{W(y|x) - c}) \nonumber \\
&= \frac{1}{\lambda_{\varphi^*}}\rho_{\varphi^*}(\mathcal{L}_{\varphi^*}(({\bf 1}_I \circ A \circ \pi_{1, 1})(\cdot, x)e^{W(\cdot|x) - c})(y)) \nonumber \\
&= \frac{1}{\lambda_{\varphi^*}}\rho_{\varphi^*}(\mathcal{L}_{\varphi}(({\bf 1}_I \circ A \circ \pi_{1, 1})(y, \cdot)e^{W(y | \cdot) - c})(x)) \nonumber \\
&= \frac{1}{\lambda_{\varphi^*}}\mathcal{L}_{\varphi}(\rho_{\varphi^*}(({\bf 1}_I \circ A \circ \pi_{1, 1})(y, \cdot)e^{W(y | \cdot) - c})))(x) \nonumber \\
&= \frac{1}{\lambda_{\varphi^*}}\mathcal{L}_{\varphi}(f)(x) \nonumber \,,
\end{align}
where the second last equality is a consequence of the Fubini's Theorem.

Therefore, $f$ is an eigenfunction of the linear operator $\mathcal{L}_{\varphi}$ associated to the eigenvalue $\lambda_{\varphi^*}$. On other hand, since the set of admissible sequences $\mathcal{B}(A, I)$ is topologically mixing, by item (4) of Theorem \ref{RPF-theorem}, it follows that the only eigenvalue for which its corresponding eigenfunctions are strictly positive is the maximal one $\lambda_{\varphi}$, which implies that $\lambda_{\varphi^*} = \lambda_{\varphi}$. Thus, $\mathcal{L}_{\varphi}(f) = \lambda_{\varphi}f$ and $\rho_{\varphi}(f) = 1$.

In a similar way, taking $\rho_{\varphi}$ as the eigenprobability associated to $\mathcal{L}^*_{\varphi}$, which is defined on the Borelian sets in $\mathcal{B}(A, I)$, it is not difficult to show that $f^*$ defined as $f^*(y) = \rho_{\varphi}(({\bf 1}_I \circ A \circ \pi_{1, 1})(y, x)e^{W(y|x) - c})$ for each $y \in \mathcal{B}(A, I)^*$, satisfies the equations $\mathcal{L}_{\varphi^*}(f^*) = \lambda_{\varphi^*}f^*$ and $\rho_{\varphi^*}(f^*) = 1$, which concludes the proof of item (1) of Theorem \ref{bilateral-extension}. 

In order to prove item (2) of this Theorem, it is enough to show that for any $\psi \in \mathcal{C}(\mathcal{B}(A, I))$ we have  $\mu_{\widehat{\varphi}}(\widehat{\psi}) = \mu_{\varphi}(\psi)$. 

Indeed, we have that \victor{
\begin{align}
\mu_{\widehat{\varphi}}(\widehat{\psi})
%&= \rho_{\varphi^*} \times \rho_{\varphi}(({\bf 1}_I \circ A \circ \pi_{1, 1})e^{W - c}\widehat{\psi}) \nonumber \\
&= \rho_{\varphi}(\rho_{\varphi^*}(({\bf 1}_I \circ A \circ \pi_{1, 1})(y, x)e^{W(y|x) - c}\widehat{\psi}(y|x))) \nonumber \\
&= \rho_{\varphi}(\psi(x)\rho_{\varphi^*}(({\bf 1}_I \circ A \circ \pi_{1, 1})(y, x)e^{W(y|x) - c})) \nonumber \\
&= \rho_{\varphi}(\psi(x)f(x)) \nonumber \\
&= \mu_{\varphi}(\psi) \nonumber \,.
\end{align}}
where we used Fubini´s theorem in the second equality and in the last one we used item (3) of Theorem \ref{RPF-theorem}. The proof that for any $\psi \in \mathcal{C}(\mathcal{B}(A, I)^*)$, $\mu_{\widehat{\varphi}}(\widehat{\psi}) = \mu_{\varphi^*}(\psi^*)$ is satisfied, is analogous to the previous case. 

Let $\widehat{\mu}_\infty$ be an accumulation point at infinity of the family $(\mu_{t \widehat{\varphi}})_{t > 1}$. Then, there is a strictly increasing sequence $(t_n)_{n \in \mathbb{N}}$ such that the following limit is satisfied in the weak* topology
\[
\lim_{n \to \infty} \mu_{t_n \widehat{\varphi}} = \widehat{\mu}_\infty \;. 
\]

Hence, taking a subsequence if necessary, we can guarantee existence of a probability measure $\mu_\infty \in \mathcal{M}_{\mathrm{max}}(\varphi)$ such that
\[
\lim_{n \to \infty} \mu_{t_n \varphi} = \mu_\infty \;. 
\]

Therefore, from the properties of the weak* topology, it follows that 
\[
\widehat{\mu}_\infty(\widehat{\varphi}) = \lim_{n \to \infty} \mu_{t_n \widehat{\varphi}}(\widehat{\varphi}) = \lim_{n \to \infty} \mu_{t_n \varphi}(\varphi) = \mu_\infty(\varphi) = m(\varphi)\;.
\]

The foregoing concludes the proof of Theorem \ref{bilateral-extension}.
\end{proof}

\section*{Acknowledgments}
The authors are very grateful to the professor Artur Oscar Lopes for his unconditional support, helpful talks and extremely useful suggestions that improved the final version of this paper. The second author would to thank to  PNPD-CAPES, INCTMat and the Francisco José de Caldas Fund by the financial support during part of the development of this paper.


\begin{thebibliography}{10}

\bibitem{MR1450400}
J.~Aaronson.
\newblock {\em An introduction to infinite ergodic theory}, volume~50 of {\em
  Mathematical Surveys and Monographs}.
\newblock American Mathematical Society, Providence, RI, 1997.
\newblock http://dx.doi.org/10.1090/surv/050.

\bibitem{MR3114331}
A.~Baraviera, R.~Leplaideur, and A.~Lopes.
\newblock {\em Ergodic optimization, zero temperature limits and the max-plus
  algebra}.
\newblock Publica\c{c}\~{o}es Matem\'{a}ticas do IMPA. [IMPA Mathematical
  Publications]. Instituto Nacional de Matem\'{a}tica Pura e Aplicada (IMPA),
  Rio de Janeiro, 2013.
\newblock 29${^{{}}{\rm{o}}}$ Col\'{o}quio Brasileiro de Matem\'{a}tica. [29th
  Brazilian Mathematics Colloquium].

\bibitem{MR2210682}
A.~Baraviera, A.~O. Lopes, and P.~Thieullen.
\newblock A large deviation principle for the equilibrium states of
  {H}\"{o}lder potentials: the zero temperature case.
\newblock {\em Stoch. Dyn.}, 6(1):77--96, 2006.
\newblock https://doi.org/10.1142/S0219493706001657.

\bibitem{MR2864625}
A.~T. Baraviera, L.~M. Cioletti, A.~O. Lopes, J.~Mohr, and R.~R. Souza.
\newblock On the general one-dimensional {$XY$} model: positive and zero
  temperature, selection and non-selection.
\newblock {\em Rev. Math. Phys.}, 23(10):1063--1113, 2011.
\newblock https://doi.org/10.1142/S0129055X11004527.

\bibitem{zbMATH03482645}
R.~{Bowen}.
\newblock {\em {Equilibrium states and the ergodic theory of Anosov
  diffeomorphisms}}, volume 470.
\newblock Springer, Cham, 1975.

\bibitem{MR1958608}
J.~Br{\'e}mont.
\newblock Gibbs measures at temperature zero.
\newblock {\em Nonlinearity}, 16(2):419--426, 2003.
\newblock http://dx.doi.org/10.1088/0951-7715/16/2/303.

\bibitem{ChFr19}
J.~Chauta and R.~Freire.
\newblock Peierls barrier for countable markov shifts.
\newblock 2019.
\newblock arXiv:1904.09655.

\bibitem{MR2818689}
J.-R. Chazottes, J.-M. Gambaudo, and E.~Ugalde.
\newblock Zero-temperature limit of one-dimensional {G}ibbs states via
  renormalization: the case of locally constant potentials.
\newblock {\em Ergodic Theory Dynam. Systems}, 31(4):1109--1161, 2011.
\newblock http://dx.doi.org/10.1017/S014338571000026X.

\bibitem{MR3656287}
L.~Cioletti, M.~Denker, A.~O. Lopes, and M.~Stadlbauer.
\newblock Spectral properties of the {R}uelle operator for product-type
  potentials on shift spaces.
\newblock {\em J. Lond. Math. Soc. (2)}, 95(2):684--704, 2017.
\newblock https://doi.org/10.1112/jlms.12031.

\bibitem{CiLo17}
L.~Cioletti and A.~O. Lopes.
\newblock Correlation inequalities and monotonicity properties of the ruelle
  operator.
\newblock {\em Stoch. and Dyn.}, 19(6):1950048, 31, 2019.
\newblock https://doi.org/10.1142/S0219493719500485.

\bibitem{CSS19}
L.~Cioletti, E.~A. Silva, and M.~Stadlbauer.
\newblock Thermodynamic formalism for topological markov chains on standard
  borel spaces.
\newblock {\em Discrete Contin. Dyn. Syst.}, 39(11):6277--6298, 2019.
\newblock https://doi.org/10.3934/dcds.2019274.

\bibitem{MR3297616}
G.~Contreras, A.~O. Lopes, and E.~R. Oliveira.
\newblock Ergodic transport theory, periodic maximizing probabilities and the
  twist condition.
\newblock In {\em Modeling, dynamics, optimization and bioeconomics. {I}},
  volume~73 of {\em Springer Proc. Math. Stat.}, pages 183--219. Springer,
  Cham, 2014.
\newblock https://doi.org/10.1007/978-3-319-04849-9\_12.

\bibitem{MR3194082}
E.~A. da~Silva, R.~R. da~Silva, and R.~R.~a. Souza.
\newblock The analyticity of a generalized {R}uelle's operator.
\newblock {\em Bull. Braz. Math. Soc. (N.S.)}, 45(1):53--72, 2014.
\newblock https://doi.org/10.1007/s00574-014-0040-3.

\bibitem{MR3864383}
R.~Freire and V.~Vargas.
\newblock Equilibrium states and zero temperature limit on topologically
  transitive countable {M}arkov shifts.
\newblock {\em Trans. Amer. Math. Soc.}, 370(12):8451--8465, 2018.
\newblock https://doi.org/10.1090/tran/7291.

\bibitem{MR2293630}
G.~Iommi.
\newblock Ergodic optimization for renewal type shifts.
\newblock {\em Monatsh. Math.}, 150(2):91--95, 2007.
\newblock http://dx.doi.org/10.1007/s00605-005-0389-x.

\bibitem{MR2151222}
O.~Jenkinson, R.~D. Mauldin, and M.~Urba{\'n}ski.
\newblock Zero temperature limits of {G}ibbs-equilibrium states for countable
  alphabet subshifts of finite type.
\newblock {\em J. Stat. Phys.}, 119(3-4):765--776, 2005.
\newblock http://dx.doi.org/10.1007/s10955-005-3035-z.

\bibitem{MR2800665}
T.~Kempton.
\newblock Zero temperature limits of {G}ibbs equilibrium states for countable
  {M}arkov shifts.
\newblock {\em J. Stat. Phys.}, 143(4):795--806, 2011.
\newblock http://dx.doi.org/10.1007/s10955-011-0195-x.

\bibitem{MR1484730}
B.~P. Kitchens.
\newblock {\em Symbolic dynamics}.
\newblock Universitext. Springer-Verlag, Berlin, 1998.
\newblock One-sided, two-sided and countable state Markov shifts.

\bibitem{MR2176962}
R.~Leplaideur.
\newblock A dynamical proof for the convergence of {G}ibbs measures at
  temperature zero.
\newblock {\em Nonlinearity}, 18(6):2847--2880, 2005.
\newblock http://dx.doi.org/10.1088/0951-7715/18/6/023.

\bibitem{LW20}
R.~Leplaideur and F.~Watbled.
\newblock Curie-weiss type models for general spin spaces and quadratic
  pressure in ergodic theory.
\newblock {\em J. Stat. Phys.}, 2020.
\newblock http://dx.doi.org/10.1007/s10955-020-02579-z.

\bibitem{LW19}
R.~Leplaideur and F.~Watbled.
\newblock Generalized curie-weiss-potts model and quadratic pressure in ergodic
  theory.
\newblock 2020.
\newblock arXiv:2003.09535.

\bibitem{MR3377291}
A.~O. Lopes, J.~K. Mengue, J.~Mohr, and R.~R. Souza.
\newblock Entropy and variational principle for one-dimensional lattice systems
  with a general {\it a priori} probability: positive and zero temperature.
\newblock {\em Ergodic Theory Dynam. Systems}, 35(6):1925--1961, 2015.
\newblock https://doi.org/10.1017/etds.2014.15.

\bibitem{LMSV19}
A.~O. Lopes, A.~Messaoudi, M.~Stadlbauer, and V.~Vargas.
\newblock Invariant probabilities for discrete time linear dynamics via
  thermodynamic formalism.
\newblock 2019.
\newblock arXiv:1910.04902.

\bibitem{MR2496111}
A.~O. Lopes, J.~Mohr, R.~R. Souza, and P.~Thieullen.
\newblock Negative entropy, zero temperature and {M}arkov chains on the
  interval.
\newblock {\em Bull. Braz. Math. Soc. (N.S.)}, 40(1):1--52, 2009.
\newblock https://doi.org/10.1007/s00574-009-0001-4.

\bibitem{1LoVa19}
A.~O. Lopes and V.~Vargas.
\newblock Gibbs states and gibbsian specifications on the space
  {$\mathbb{R}^{\mathbb{N}}$}.
\newblock {\em Dyn. Systems: an Int. Jour.}, 35(2):216--241, 2020.
\newblock https://doi.org/10.1080/14689367.2019.1663789.

\bibitem{2LoVa19}
A.~O. Lopes and V.~Vargas.
\newblock The ruelle operator for symmetric {$\beta$}-shifts.
\newblock {\em Publ. Mat.}, 64(2):661--680, 2020.
\newblock https://doi.org/10.5565/PUBLMAT6422012.

\bibitem{MR1853808}
R.~D. Mauldin and M.~Urba{\'n}ski.
\newblock Gibbs states on the symbolic space over an infinite alphabet.
\newblock {\em Israel J. Math.}, 125:93--130, 2001.
\newblock http://dx.doi.org/10.1007/BF02773377.

\bibitem{MR2003772}
R.~D. Mauldin and M.~Urba{\'n}ski.
\newblock {\em Graph directed {M}arkov systems}, volume 148 of {\em Cambridge
  Tracts in Mathematics}.
\newblock Cambridge University Press, Cambridge, 2003.
\newblock Geometry and dynamics of limit sets.

\bibitem{Moh18}
J.~Mohr.
\newblock Product type potential on the xy model: Selection of maximizing
  probability and a large deviation principle.
\newblock 2019.
\newblock arXiv:1805.09858.

\bibitem{MR1085356}
W.~Parry and M.~Pollicott.
\newblock Zeta functions and the periodic orbit structure of hyperbolic
  dynamics.
\newblock {\em Ast\'erisque}, (187-188):268, 1990.

\bibitem{MR0234697}
D.~Ruelle.
\newblock Statistical mechanics of a one-dimensional lattice gas.
\newblock {\em Comm. Math. Phys.}, 9:267--278, 1968.

\bibitem{MR511655}
D.~Ruelle.
\newblock {\em Thermodynamic formalism}, volume~5 of {\em Encyclopedia of
  Mathematics and its Applications}.
\newblock Addison-Wesley Publishing Co., Reading, Mass., 1978.
\newblock The mathematical structures of classical equilibrium statistical
  mechanics, With a foreword by Giovanni Gallavotti and Gian-Carlo Rota.

\bibitem{MR1955261}
O.~Sarig.
\newblock Existence of {G}ibbs measures for countable {M}arkov shifts.
\newblock {\em Proc. Amer. Math. Soc.}, 131(6):1751--1758 (electronic), 2003.
\newblock http://dx.doi.org/10.1090/S0002-9939-03-06927-2.

\bibitem{MR1738951}
O.~M. Sarig.
\newblock Thermodynamic formalism for countable {M}arkov shifts.
\newblock {\em Ergodic Theory Dynam. Systems}, 19(6):1565--1593, 1999.
\newblock http://dx.doi.org/10.1017/S0143385799146820.

\bibitem{MR691854}
Y.~G. Sina\u{\i}.
\newblock {\em Theory of phase transitions: rigorous results}, volume 108 of
  {\em International Series in Natural Philosophy}.
\newblock Pergamon Press, Oxford-Elmsford, N.Y., 1982.
\newblock Translated from the Russian by J. Fritz, A. Kr\'{a}mli, P. Major and
  D. Sz\'{a}sz.

\bibitem{zbMATH05995684}
C.~J. {Thompson}.
\newblock {Infinite-spin Ising model in one dimension}.
\newblock {\em {J. Math. Phys.}}, 9:241--245, 1968.
\newblock https://doi.org/10.1063/1.1664574.

\bibitem{MR2316198}
A.~C.~D. van Enter and W.~M. Ruszel.
\newblock Chaotic temperature dependence at zero temperature.
\newblock {\em J. Stat. Phys.}, 127(3):567--573, 2007.
\newblock http://dx.doi.org/10.1007/s10955-006-9260-2.

\end{thebibliography}
\end{document}